\documentclass[]{article}

\usepackage{dsfont}
\usepackage{graphicx}
\usepackage{amsmath}
\usepackage{amssymb}
\usepackage{amsthm}
\usepackage{mathrsfs}
\usepackage{pxfonts}
\usepackage{enumerate}
\usepackage{color}
\usepackage{mathdots}
\usepackage{sectsty}
\usepackage{tikz}
\usepackage{adjustbox}
\usepackage{enumitem}
\usepackage{caption}
\usepackage{bbold}
\usepackage[hidelinks]{hyperref}
\allowdisplaybreaks
\sectionfont{\scshape\centering\fontsize{11}{14}\selectfont}
\subsectionfont{\scshape\fontsize{11}{14}\selectfont}
\usepackage{fancyhdr}
\usepackage[nottoc,notlot,notlof]{tocbibind}

\newcommand\shorttitle{C$\beta$E moments, $\mathsf{Sine}_\beta$ correlations and stochastic zeta}
\newcommand\authors{T. Assiotis and J. Najnudel}

\newtheorem{thm}{Theorem}[section]
\newtheorem{cor}[thm]{Corollary}
\newtheorem{lem}[thm]{Lemma}
\newtheorem{defn}[thm]{Definition}
\newtheorem{rmk}[thm]{Remark}
\newtheorem{prop}[thm]{Proposition}

\newtheorem{conj}[thm]{Conjecture}

\newtheorem*{theorem*}{Theorem}

\title{\large \bf MOMENTS OF C$\beta$E FIELD PARTITION FUNCTION, $\mathsf{Sine}_{\beta}$ CORRELATIONS AND STOCHASTIC ZETA}
\author{\small THEODOROS ASSIOTIS AND JOSEPH NAJNUDEL}
\date{}

\begin{document}

\maketitle

\begin{abstract} 
We prove a conjecture of Fyodorov and Keating on  the supercritical moments of the partition function of the C$\beta$E field or equivalently the supercritical moments of moments of the characteristic polynomial of the C$\beta$E ensemble for general $\beta>0$ and general real moment exponents. Moreover, we give the first expression for all correlation functions of the $\mathsf{Sine}_\beta$ point process for all $\beta>0$. The main object behind both results is the Hua-Pickrell stochastic zeta function introduced by Li and Valk\'{o}.
\end{abstract}

\tableofcontents

\section{Introduction}

The purpose of this paper is to solve two seemingly different problems in random matrix theory: prove a conjecture of Fyodorov and Keating on moments of characteristic polynomials of the circular $\beta$ ensemble in Theorem \ref{MainThm1} and give a first expression for all $\beta>0$ of all correlation functions of the $\mathsf{Sine}_\beta$ point process in Theorem \ref{MainThm2}. Although these may seem quite different, the basic object behind both statements, the Hua-Pickrell stochastic zeta function of Li and Valk\'{o}, and also the proof technique, based on random orthogonal polynomials on the unit circle, are the same.

We begin the article by discussing the general background.

\subsection{Moments of the partition function of the C$\beta$E field}

Let $\mathbb{T}$ be the unit circle that we identify with $[-\pi,\pi)$. Define the $\textnormal{C}\beta \textnormal{E}_N$ ensemble to be the probability measure on $\mathbb{T}^N$ given by:
\begin{equation*}
\frac{1}{\mathcal{Z}_{N,\beta}}\prod_{1\le j <k \le N} \left|\mathrm{e}^{\textnormal{i}\theta_j}-\mathrm{e}^{\textnormal{i}\theta_k}\right|^\beta \mathrm{d}\theta_1\cdots \mathrm{d}\theta_N, \ \  \ \mathcal{Z}_{N,\beta}=(2\pi)^N \frac{\Gamma\left(\beta\frac{N}{2}+1\right)}{\Gamma\left(\frac{\beta}{2}+1\right)}.
\end{equation*}
We denote by $\mathbb{E}_{\textnormal{C}\beta \textnormal{E}_N}$ expectation with respect to it. For $\beta=2$ this is the law of eigenvalues of a Haar distributed random unitary matrix and is known as the Circular Unitary Ensemble (CUE), see \cite{ForresterBook}. The choices $\beta=1, 4$ also have a natural conjugation-invariant random matrix interpretation, see \cite{ForresterBook}.

Define the characteristic polynomial $\mathsf{X}_N(z)$, with $(\mathrm{e}^{\mathrm{i}\theta_j})_{j=1}^N$ $\textnormal{C}\beta \textnormal{E}_N$-distributed, by
\begin{equation*}
\mathsf{X}_N(z)\overset{\mathrm{def}}{=}\prod_{j=1}^N\left(1-z\mathrm{e}^{-\mathrm{i}\theta_j}\right)
\end{equation*}
and the  $\textnormal{C}\beta \textnormal{E}_N$ field $\mathsf{L}_N(\theta)\overset{\mathrm{def}}= \log |\mathsf{X}_N(\mathrm{e}^{\mathrm{i}\theta})|$ on $\mathbb{T}$. 

\begin{defn}
Let $\beta>0$, $k, s \in \mathbb{R}_+$. Define the moments of the partition function of $\mathsf{L}_N$ or moments of moments of $\mathsf{X}_N$, where the points $\{\mathrm{e}^{\mathrm{i\theta_j}}\}$  are $\textnormal{C}\beta \textnormal{E}_N$-distributed, by 
\begin{equation*}
\mathsf{M}^{(\beta)}_N(k;s)\overset{\mathrm{def}}{=}\mathbb{E}_{\textnormal{C}\beta \textnormal{E}_N}\left[\left(\frac{1}{2\pi}\int_{-\pi}^\pi\mathrm{e}^{2s\mathsf{L}_N(\theta)}\mathrm{d}\theta\right)^{k}\right]= \mathbb{E}_{\textnormal{C}\beta \textnormal{E}_N}\left[\left(\frac{1}{2\pi}\int_{-\pi}^\pi\left|\mathsf{X}_N(\mathrm{e}^{\mathrm{i}\theta})\right|^{2s}\mathrm{d}\theta\right)^{k}\right].
\end{equation*} 
\end{defn}
These quantities were first considered by Fyodorov and Keating \cite{FyodorovKeating} (first in the case $\beta=2$ of CUE, then for general $\beta>0$ \cite{FyodorovGnutzmannKeating,KeatingWong}) who made various predictions about their asymptotic behaviour as $N \to \infty$ depending on the values of $k, s, \beta$. Such predictions have then been developed to predictions in number theory, for generalised moments of the Riemann zeta function, by Bailey and Keating in \cite{BaileyKeatingZeta}, following the Keating-Snaith philosophy \cite{KeatingSnaith,KeatingSnaithSympOrth}.

$\bullet$ \textbf{The GMC regime.} In a certain parameter range these asymptotics are conceptually well-understood owing to a connection to log-correlated Gaussian fields and the Gaussian multiplicative chaos (GMC), see \cite{BaileyKeatingSurvey,KeatingWong}. It is known that, with convergence in distribution (denoted by $\overset{\mathrm{d}}{\longrightarrow}$) taking place in a negative Sobolev space, see \cite{HughesKeatingOConnell},
\begin{equation}\label{FieldConv}
\mathsf{L}_N(\bullet) \overset{\textnormal{d}}{\longrightarrow}   \beta^{-\frac{1}{2}}\mathsf{G}(\bullet),
\end{equation}
where $\mathsf{G}$ is the Gaussian free field on $\mathbb{T}$ with covariance: 
\begin{equation*}
\mathbb{E}\left[\mathsf{G}(x)\mathsf{G}(y)\right]=-\log\left|\mathrm{e}^{\textnormal{i}x}-\mathrm{e}^{\textnormal{i}y}\right|.
\end{equation*}
From $\mathsf{G}$ one can define, for $\gamma< \sqrt{2}$, a non-trivial random measure $\mathsf{GMC}_\gamma$ on $\mathbb{T}$ given informally by the expression:
\begin{equation*}
\mathsf{GMC}_\gamma(\mathrm{d}\theta)=``\frac{\mathrm{e}^{\gamma \mathsf{G}(\theta)}}{\mathbb{E}\left[e^{\gamma \mathsf{G}(\theta)}\right]}\mathrm{d}\theta=\mathrm{e}^{\gamma \mathsf{G}(\theta)-\frac{\gamma^2}{2}\mathbb{E}\left[\mathsf{G}^2(\theta)\right]}\mathrm{d}\theta",
\end{equation*}
and this is called the GMC on $\mathbb{T}$, see \cite{RhodesVargas,Berestycki}. Hence, from \eqref{FieldConv} it is natural to expect that for $2s^2<\beta$ one should have the following weak convergence of random measures:
\begin{equation}\label{GMCconvergence}
  \frac{\left|\mathsf{X}_N(\mathrm{e}^{\textnormal{i}\theta})\right|^{2s}}{\mathbb{E}_{\textnormal{C}\beta \textnormal{E}_N}\left[\left|\mathsf{X}_N(\mathrm{e}^{\textnormal{i}\theta})\right|^{2s}\right]}  \mathrm{d}\theta\overset{\textnormal{d}}{\longrightarrow} \mathsf{GMC}_{2s\beta^{-\frac{1}{2}}}\left(\mathrm{d}\theta\right).
\end{equation}
This was first proven in the case $\beta=2$ of CUE using Riemann-Hilbert techniques \cite{Webb,NikulaSaksmanWebb} and recently established for general $\beta>0$ \cite{LambertNajnudel} using a probabilistic approach. The critical case $\beta=2s^2$, for which a different renormalisation is required, is of particular interest because of its relation to the following conjecture of Fyodorov-Hiary-Keating \cite{FyodorovKeating,FyodorovHiaryKeating} on the maximum of the field $\mathsf{L}_N$ on $\mathbb{T}$, as $N \to \infty$:
\begin{equation}\label{MaxConjecture}
\max_{\theta \in \mathbb{T}}\mathsf{L}_N(\theta)-\sqrt{\frac{2}{\beta}} \left(\log N-\frac{3}{4}\log \log N\right) \overset{\textnormal{d}}{\longrightarrow} \mathfrak{G}_\beta,
\end{equation}
where $\mathfrak{G}_\beta$ is an explicit random variable, closely related to critical GMC, see \cite{ArguinBeliusBourgade,PaquetteZeitouni1,ChhaibiMadauleNajnudel,PaquetteZeitouni2}  for progress. An analogous conjecture exists for the Riemann $\zeta$-function on the critical line $\Re(z)=\frac{1}{2}$, motivated by the above random matrix prediction, see \cite{FyodorovKeating,FyodorovHiaryKeating,NajnudelZeta,ABBRS,Harper,ABR1,ABR2}.

Back to $\mathsf{M}^{(\beta)}_N(k;s)$, with $2s^2<\beta$ so that \eqref{GMCconvergence} holds, known as the moment-subcritical regime $2ks^2<\beta$, $\mathsf{M}^{(\beta)}_N(k;s)$ grows like a multiplicative constant times $N^{\frac{2}{\beta}ks^2}$, see \cite{Remy,KeatingWong,LambertNajnudel}. In this case, the leading order coefficient is basically given by the moments of the total mass of the GMC, see \cite{Remy,KeatingWong,LambertNajnudel} and the discussion in Remark \ref{FormulaComparison} below. These moments have now been computed rigorously, see \cite{Remy,ChhaibiNajnudel} and \cite{FyodorovBouchaud} for the original physics prediction of Fyodorov-Bouchaud. In the moment-critical regime $\beta=2ks^2$, $\mathsf{M}^{(\beta)}_N(k;s)$ is conjectured by Keating and Wong \cite{KeatingWong} to grow like $N\log N$ times an explicit leading order coefficient again related to GMC, see \cite{KeatingWong}. For $\beta=2$, $k\in \mathbb{N}$, Keating and Wong proved their conjecture in the same paper. 

$\bullet$ \textbf{The supercritical regime.} When $2ks^2>\beta$, the connection to moments of the GMC breaks down and it is not as clear how  and why $\mathsf{M}^{(\beta)}_N(k;s)$ should behave asymptotically. Nevertheless, the following remarkable prediction was made in \cite{FyodorovKeating,FyodorovGnutzmannKeating, KeatingWong}, see Conjecture 2.4 in \cite{BaileyKeatingSurvey} for the special case $\beta=2$.

\begin{conj}\label{prediction} 
Let $\beta>0, k\in \mathbb{N}, s\in \mathbb{R}_+$. For $2ks^2>\beta$, as $N \to \infty$,
\begin{equation}
\mathsf{M}^{(\beta)}_N(k;s) \sim \mathfrak{c}^{(\beta)}(k;s) N^{2k^2s^2\beta^{-1}-k+1}.
\end{equation}
\end{conj}
The conjecture also makes sense for real exponents $k$ and we will establish that as well for a certain parameter range. The leading order coefficient $\mathfrak{c}^{(\beta)}(k;s)$ was unknown at the time Conjecture \ref{prediction} was first made.  Since then there has been significant work on this problem using analytic and combinatorial techniques.  Various partial results have been obtained with different expressions given for $\mathfrak{c}^{(\beta)}(k;s)$, ranging from representations in terms of Painlev\'{e} equations to volumes of certain polytopes, that we now survey. 

Firstly, for $k=1$, $\mathsf{M}^{(\beta)}_N(1;s)$ can be computed completely explicitly using the Selberg integral and the $N \to \infty$ asymptotics can readily be established, see \cite{KeatingSnaith,RandomDets}. For $k\neq 1$, work on Conjecture \ref{prediction} mainly focused on the case $\beta=2$ of CUE for which a number of tools are available:

$\bullet  \boldsymbol{\; \beta=2}$. The asymptotics for $k=2$, $s \in \mathbb{R}_+$, $\beta=2$ were first established by Claeys and Krasovsky in \cite{ClaeysKrasovsky} using Riemann-Hilbert problem techniques who also gave a representation of $\mathfrak{c}^{(2)}(2;s)$ in terms of the Painlev\'{e} V equation. This follows from computing the asymptotics of $N\times N$ Toeplitz determinants with symbols with two merging Fisher-Hartwig singularities, a problem with very long history, see \cite{Widom,ClaeysKrasovsky,ClaeysFahs,Fahs,BourgadeFalconet}. Two alternative proofs of the asymptotics for $k=2$, $s \in \mathbb{N}$, $\beta=2$ were given in \cite{KRRR}. One proof makes use of a multiple contour integral representation coming from \cite{CFKRS} and the other one is algebraic combinatorial making use of symmetric function theory, see \cite{BympGamburd}. The complex analytic and combinatorial proofs were then extended in \cite{BaileyKeatingMoM} and \cite{AssiotisKeatingMoM} respectively to prove the asymptotics for general $k\in \mathbb{N}$, $s \in \mathbb{N}$, $\beta=2$. The combinatorial proof gives an expression for $\mathfrak{c}^{(2)}(k;s)$ in terms of the volume of a Gelfand-Tsetlin type polytope. This expression, for $k=2$, can be used to give an alternative proof of the representation of $\mathfrak{c}^{(2)}(2;s)$ in terms of the Painlev\'{e} V equation \cite{BasorGeRubinstein,BasorChenEhrhardt}. For applications of these methods to allied problems, see \cite{ABK, MoMEhrhart,BaileyKeatingSurvey}. Finally, for $k\in  \mathbb{N}$ and $s\in \mathbb{R}_+$ so that $ks^2>1$, Fahs in \cite{Fahs} using Riemann-Hilbert problem techniques gave the following upper and lower bound:
\begin{equation*}
\mathsf{M}^{(2)}_N(k;s)=\mathrm{e}^{\mathcal{O}_{k,s}(1)} N^{2k^2s^2\beta^{-1}-k+1}.
\end{equation*}
Obtaining some expression for the $\mathrm{e}^{\mathcal{O}_{k,s}(1)}$ term above when $s \in \mathbb{R}_+$, $k>2$,  and further connecting it to integrable systems is a well-known folklore problem in the Riemann-Hilbert problem community. It boils down to a delicate understanding of asymptotics of $N\times N$ Toeplitz determinants with multiple (growing in number with $k$) merging  Fisher-Hartwig singularities that goes beyond Fahs' tour-de-force work \cite{Fahs}. 

$\bullet \boldsymbol{\; \beta \neq 2}$. For general $\beta \neq 2$ very little was known. Partial results, for $k,s \in \mathbb{N}$, were first obtained in \cite{MoMCbetaE}, see also more recently \cite{forrester2025dualities}, using combinatorics. The asymptotics were established for $k=2$, $s \in \mathbb{N}$, $4s^2>\beta$ and for $k\in \mathbb{N}$, $k > 2$, $s\in \mathbb{N}$, $\beta \le 2$. The leading order coefficient was given as an integral over the polytope mentioned above. The integrand can be written in terms of the Dixon-Anderson probability distribution, see \cite{ForresterBook}.

It was until now unclear whether a natural probabilistic object existed which would replace GMC in the supercritical regime. The results above did not give any indication. Our first main result, Theorem \ref{MainThm1} below, establishes the asymptotics of $\mathsf{M}^{(\beta)}_N(k;s)$ in the supercritical regime for general $\beta>0$ and generic real moment exponents $k, s$, in particular proving Conjecture \ref{prediction}, and reveals what this object is. Our approach is different from previous works and is probabilistic in nature.

\subsection{$\mathsf{Sine}_{\beta}$ correlations }

The general $\beta>0$, $\mathsf{Sine}_\beta$ point process, arguably one of the most fundamental objects in random matrix theory, is the universal scaling limit of general $\beta$-ensembles: see \cite{ValkoViragCarousel,KillipStoiciu,BourgadeBulk}. It was first constructed for all $\beta>0$ by Valk\'{o} and Virag  \cite{ValkoViragCarousel} and around the same time by Killip and Stoiciu \cite{KillipStoiciu}. The special case $\mathsf{Sine}_2$ is the well-known determinantal point process \cite{BorodinDet,JohanssonDet,ForresterBook} with correlations given by the famous sine kernel \cite{ForresterBook}. The other classical values $\beta=1,4$ also possess special algebraic structure, with correlations now given in terms of Pfaffians, see \cite{ForresterBook}. 

We now give the precise definition of $\mathsf{Sine}_\beta$, see \cite{ValkoViragCarousel,ValkoViragOperators,KillipStoiciu}, via its counting function. Equivalent characterisations exist, such as being the spectrum of a certain random Dirac operator, see for example \cite{ValkoViragOperators}. Define $(\mathsf{w}_\mathbb{C}(t))_{t\in (-\infty,\infty)}$ to be a two-sided complex Brownian motion. Then, consider the unique strong solution $(\mathsf{p}_x^{(\beta)}(t))_{t\in (-\infty,\infty)}$ to the one-parameter, in $\lambda \in \mathbb{R}$, coupled (as they all share $\mathsf{w}_{\mathbb{C}}$) family of stochastic differential equations (SDE),
\begin{equation}\label{SineSDE}
\mathrm{d}\mathsf{p}_\lambda^{(\beta)}(t)=\lambda \frac{\beta}{4}\mathrm{e}^{\frac{\beta}{4}t}\mathrm{d}t+\Re\left[\left(\mathrm{e}^{-\mathrm{i}\mathsf{p}_\lambda^{(\beta)}(t)}-1\right)\mathrm{d}\mathsf{w}_\mathbb{C}(t)\right],
\end{equation}
with $t \in (-\infty,\infty)$,
and with initial condition $\lim_{t\to -\infty}\mathsf{p}^{(\beta)}_{\lambda}(t)=0$. Finally, let $\mathsf{U}$ be a uniform random variable on $[0,2\pi)$ independent of $\mathsf{w}_\mathbb{C}$. 

\begin{defn} 
For $\beta>0$, the $\mathsf{Sine}_\beta$ point process is characterised as
\begin{equation*}
\mathsf{Sine}_\beta\overset{\mathrm{d}}{=} \left\{\lambda \in \mathbb{R}: \mathsf{p}^{(\beta)}_\lambda(0)=\mathsf{U} \ \mathrm{mod} \ 2\pi\right\}.
\end{equation*}
\end{defn}
Here and throughout the paper $\overset{\mathrm{d}}{=}$ denotes equality in distribution.

Moving on to correlations, for a simple point process $\boldsymbol{\Xi}$ on $\mathbb{R}$ we denote by $\boldsymbol{\Xi}(\mathscr{A})$ the number of points of $\boldsymbol{\Xi}$ in $\mathscr{A}$, where $\mathscr{A}$ is a Borel subset of $\mathbb{R}$. The $m$-th correlation function of $\boldsymbol{\Xi}$, if it exists, is the symmetric function $\boldsymbol{\rho}^{(m)}:\mathbb{R}^m\to \mathbb{R}_+$ that satisfies: for $\mathscr{A}_1,\dots,\mathscr{A}_m$ disjoint Borel sets,
\begin{equation*}
\mathbb{E}\left[\prod_{j=1}^m \boldsymbol{\Xi}(\mathscr{A}_j)\right]=\int_{\mathscr{A}_1}\cdots\int_{\mathscr{A}_m}\boldsymbol{\rho}^{(m)}(x_1,\dots,x_m)\mathrm{d}x_1\cdots\mathrm{d}x_m.
\end{equation*}

The fact that correlation functions of all orders for $\mathsf{Sine}_\beta$ exist follows from moment bounds from \cite{ValkoViragCarousel}, see the discussion in Section 2.1 of \cite{qu2025pair}. In particular, the following definition makes sense.
\begin{defn}
For $\beta>0$ and $m\in \mathbb{N}$ we define $\boldsymbol{\rho}_\beta^{(m)}$ to be the $m$-th correlation function of $\mathsf{Sine}_\beta$.
\end{defn}
The problem of giving an expression for $\boldsymbol{\rho}_\beta^{(m)}$ and obtaining some information out of it is a long-standing one. There has been very significant body of work by Forrester in the case of even integer $\beta=2n$. Forrester using generalised hypergeometric functions and Jack polynomial theory obtains a multiple integral expression for the correlations and also relates them to higher order differential equations, see \cite{ForresterNuclearPhys,ForresterNuclearPhysAdd,ForresterConj,ForresterBook,ForresterRains,ForresterRahman,ForresterDifferential}. Moreover, a certain expression for the correlations for rational values of $\beta$ was obtained in the unpublished manuscript by Okounkov \cite{Okounkov}.

It was only very recently however, that Qu and Valk\'{o} \cite{qu2025pair} obtained the first expression for the pair correlation $\boldsymbol{\rho}_\beta^{(2)}(0,x)$ of $\mathsf{Sine}_\beta$ for general $\beta>0$. This expression is given in terms of expectations of generalisations of the diffusions from \eqref{SineSDE} above. We recall their formula, and compare to the one we obtain, in \eqref{ComparisonQuValko}. Using it they can remarkably compute precise large $x$ asymoptotics of $\boldsymbol{\rho}_\beta^{(2)}(0,x)$. Moreover, for $\beta=2n$ they can get a certain power series representation for $\boldsymbol{\rho}_\beta^{(2)}(0,x)$.

Our second main result Theorem \ref{MainThm2} is an expression for all order correlation functions of $\mathsf{Sine}_\beta$ for all $\beta>0$.  In some sense this gives implicitly a positive answer to Problem 7 of \cite{qu2025pair} which asks for an expression of higher order correlation functions of $\mathsf{Sine}_\beta$ in terms of diffusions like \eqref{SineSDE}. We also expect that the expression we obtain for the correlations will be useful in problems of infinite-dimensional dynamics related to random matrices, see \cite{Osada1,Osada2,Suzuki1,Suzuki2}.

\subsection{Main results}

In order to state our main results we need to introduce certain quantities from \cite{ValkoViragOperators,LiValko}. For $\delta>0$ consider the unique strong solution $(\mathsf{p}_\lambda^{(\beta,\delta)}(t))_{t\in (-\infty,\infty)}$ to the generalisation of the coupled SDE \eqref{SineSDE} (which corresponds to $\delta=0$), indexed by $\lambda\in \mathbb{R}$,
\begin{equation}
\mathrm{d}\mathsf{p}_\lambda^{(\beta,\delta)}(t)=\lambda \frac{\beta}{4}\mathrm{e}^{\frac{\beta}{4}t}\mathrm{d}t+\Re\left[\left(\mathrm{e}^{-\mathrm{i}\mathsf{p}_\lambda^{(\beta,\delta)}(t)}-1\right)\left(\mathrm{d}\mathsf{w}_\mathbb{C}(t)-\mathrm{i}\delta\mathrm{d}t\right)\right],
\end{equation}
 with $t \in (-\infty,\infty)$ and initial condition $\lim_{t \to -\infty} \mathsf{p}_\lambda^{(\beta,\delta)}(t)=0$. Let $\boldsymbol{\Theta}_\delta$ be a random variable on $[0,2\pi)$ with probability density:
\begin{equation*}
\frac{1}{2\pi} \frac{\Gamma(1+\delta)^2}{\Gamma(1+2\delta)} \left|1-\mathrm{e}^{\mathrm{i}\theta}\right|^{2\delta}
\end{equation*}
independent of the two-sided complex Brownian motion $\mathsf{w}_{\mathbb{C}}$. Observe that $\boldsymbol{\Theta}_0\overset{\textnormal{d}}{=}\mathsf{U}$.
\begin{defn}
Define the Hua-Pickrell point process $\mathcal{HP}_{\beta,\delta}$ with parameters $\beta,\delta >0$:
\begin{equation}
\mathcal{HP}_{\beta,\delta}\overset{\mathrm{d}}{=}\left\{\lambda \in \mathbb{R}:\mathsf{p}^{(\beta,\delta)}_\lambda(0)=\boldsymbol{\Theta}_\delta \ \mathrm{mod} \ 2\pi\right\}.
\end{equation}
\end{defn}
Note that, $\mathcal{HP}_{\beta,0}\overset{\textnormal{d}}{=}\mathsf{Sine}_\beta$. $\mathcal{HP}_{\beta,\delta}$ can in fact be considered for $\Re(\delta)>-1/2$ but we will not need this more general parameter range in this paper. For $\beta=2$, it is well-known that $\mathcal{HP}_{2,\delta}$ is determinantal with an explicit correlation kernel generalising the sine kernel, see \cite{BorodinOlshanski,ForresterWitte}. Like $\mathsf{Sine}_\beta$, $\mathcal{HP}_{\beta,\delta}$ has multiple characterisations, including as the spectrum of a more general random Dirac operator, see \cite{ValkoViragOperators,LiValko}.

We now define the Hua-Pickrell stochastic zeta function $\boldsymbol{\xi}^{\beta,\delta}$ of Li and Valk\'{o} \cite{LiValko}. We note that Li-Valk\'{o} gave a different equivalent representation of it as a power series, see \cite{LiValko}, and to get the principal value product in Definition \ref{StochasticZeta} additional arguments are required, see Proposition \ref{PropPVRep}. The special case, with $\delta=0$, corresponding to $\mathsf{Sine}_\beta$ was studied earlier by Valk\'{o} and Virag \cite{ValkoVirag} while the determinantal case of $\mathsf{Sine}_2$ was introduced and studied by Chhaibi-Najnudel-Nikeghbali in \cite{CNN}.

\begin{defn}\label{StochasticZeta} Let $\beta>0, \delta>0$. Define the following random entire function
\begin{equation}\label{PVrep}
\boldsymbol{\xi}^{\beta,\delta}(z)=\lim_{R\to \infty} \prod_{\substack{\mathsf{x} \in \mathcal{HP}_{\beta,\delta} \\ |\mathsf{x}|<R}} \left(1-\frac{z}{\mathsf{x}}\right).
\end{equation}
\end{defn}
The fact that almost surely the principal value product in \eqref{PVrep} converges uniformly on compact sets in $z \in \mathbb{C}$ will be a consequence of the proof of Proposition \ref{PropPVRep}.

We will also need the special function
\begin{align*}
\mathcal{Y}_\beta(z)&=\frac{\beta}{2}\log G\left(1+\frac{2z}{\beta}\right)-\left(z-\frac{1}{2}\right)\log \Gamma\left(1+\frac{2z}{\beta}\right)+\int_0^\infty \left(\frac{1}{2x}-\frac{1}{x^2}+\frac{1}{x(\mathrm{e}^x-1)}\right)\frac{\mathrm{e}^{-xz}-1}{\mathrm{e}^{x\beta/2}-1}\mathrm{d}x\\ &+ \frac{z^2}{\beta}+\frac{z}{2},
\end{align*}
where $G$ is the Barnes G-function. For special values of $\beta$, including $\beta \in 2 \mathbb{N}$, $\mathcal{Y}_\beta(z)$ has a much simpler expression, see Lemma 7.1 of \cite{RandomDets}. Write, for $a\in \mathbb{R}$, $\mathbb{R}_{\ge a}=[a,\infty)$.

The following is the first main result of the paper.

\begin{thm}\label{MainThm1}
 Let $k\in \mathbb{N}, s\in \mathbb{R}_+$, $\beta >0$ such that $2ks^2>\beta$. Then, as $N \to \infty$,
 \begin{equation}\label{MoMLimit}
\frac{1}{N^{2k^2s^2\beta^{-1}-k+1}}\mathsf{M}^{(\beta)}_N(k;s)\longrightarrow \mathcal{F}^{(\beta)}_{k;s}\mathbb{E}\left[\left(\int_{-\infty}^\infty \left|\boldsymbol{\xi}^{\beta,ks}(x)\right|^{2s}\mathrm{d}x\right)^{k-1}\right],
 \end{equation}
where $\mathcal{F}^{(\beta)}_{k;s}$ is given explicitly 
 \begin{equation}
    \mathcal{F}^{(\beta)}_{k;s}=(2\pi)^{1-k}\mathrm{e}^{\mathcal{Y}_\beta\left(1+2ks-\beta/2\right)-2\mathcal{Y}_\beta \left(1+ks-\beta/2\right)+\mathcal{Y}_\beta(1-\beta/2)}.
 \end{equation}
 The same result holds for $k \in \mathbb{R}_{\ge 1}$ whenever both $2s^2(2k-\lceil k \rceil)>\beta$ and, for $k>1$, $4\lceil k-1\rceil(k-\lceil k-1\rceil)s^2>\beta$.
\end{thm}
The result is optimal for $k \in \mathbb{N}$ and completely proves Conjecture \ref{prediction}. The restriction for $k \in \mathbb{R}_{\ge 1}\backslash\mathbb{N}$ is technical and due to the fact that in the proof we compare to the next larger integer moment which we can control explicitly. This also explains the fact that both restrictions are worse when $k$ is just above an integer while they get closer to optimal as $k$ approaches an integer from below. We note that before our work Conjecture \ref{prediction} had not been proven for any non-integer value of the exponent $k$, even when $\beta=2$.

By comparing to previous results in the literature \cite{ClaeysKrasovsky,KRRR, AssiotisKeatingMoM, BaileyKeatingMoM,BaileyKeatingSurvey,Fahs,MoMCbetaE} one readily obtains expressions for the integer moments of integrals of $ |\boldsymbol{\xi}^{2,ks}|^{2s}$ in terms of Painlev\'{e} equations and volumes of polytopes. It would be interesting and non-trivial to obtain them directly from $\boldsymbol{\xi}^{\beta,ks}$. Going further, it would be very interesting to use expression \eqref{MoMLimit} along with the various descriptions of $\boldsymbol{\xi}^{\beta,\delta}$ to obtain new explicit information about the limit. We mention in passing that joint moments of the Taylor coefficients of $\boldsymbol{\xi}^{2,\delta}$, for suitable $\delta$, also describe the limiting joint moments of characteristic polynomials of random unitary matrices along with their derivatives of any order, see \cite{AKW,ABGS,AGKW}.

Finally, we expect that an analogous strategy to the one we employ here for Theorem \ref{MainThm1}, along with additional technical innovations, could be used to tackle the problem of asymptotics of supercritical moments of moments for the orthogonal and unitary symplectic groups and for the Gaussian unitary ensemble (and its $\beta$-ensemble version) for which much less is known, see \cite{FyodorovSimm,ClaeysGlesnerMinakovYang,ClaeysForkelKeating}.

We now move to our next main result on the correlation functions $\boldsymbol{\rho}^{(m)}_\beta$ of $\mathsf{Sine}_\beta$.

\begin{thm}\label{MainThm2} For all $\beta>0$ and $m\in \mathbb{N}$, we have:
\begin{equation}\label{CorrFormula}
\boldsymbol{\rho}^{(m)}_\beta(x_1,\dots,x_m)=\mathfrak{C}_{\beta}^{(m)} \prod_{1\le i<j \le m}\left|x_i-x_j\right|^\beta \mathbb{E}\left[\prod_{j=2}^m \left|\boldsymbol{\xi}^{\beta,m\beta/2}\left(x_j-x_1\right)\right|^\beta\right],
\end{equation}
where the constant $\mathfrak{C}_{\beta}^{(m)}$ is explicitly given by:
\begin{equation}
\mathfrak{C}_{\beta}^{(m)}=\frac{2^{m\left(\frac{\beta}{2}-1\right)}}{\pi^m \beta^{\beta m/2}}\mathrm{e}^{\mathcal{Y}_\beta\left(1+\beta(m-\frac{1}{2})\right)-2\mathcal{Y}_\beta\left(1+\frac{\beta}{2}(m-1)\right)+\mathcal{Y}_\beta\left(1-\frac{\beta}{2}\right)}.
\end{equation}
The correlation functions $\boldsymbol{\rho}^{(m)}_\beta$ are uniformly bounded on $\mathbb{R}^m$. Moreover, the function 
\begin{equation}\label{CorrFnFn}
(x_1,\dots,x_m)\mapsto \mathbb{E}\left[\prod_{j=2}^m \left|\boldsymbol{\xi}^{\beta,m\beta/2}\left(x_j-x_1\right)\right|^\beta\right]
\end{equation}
is continuous and strictly positive for all $(x_1,\dots,x_m)\in \mathbb{R}^m$.
\end{thm}

A number of comments are in order. Firstly, using \eqref{CorrFormula} it should be possible to prove  continuity of $\beta\mapsto \boldsymbol{\rho}_\beta^{(m)}(x_1,\dots,x_m)$. However, this requires additional technical efforts and we will not attempt to do this here. Moreover, it is easy to see from \eqref{CorrFormula} that the even integer $\beta=2n$ case is special. One can expand the $\beta=2n$ powers of the entire function $z\mapsto\boldsymbol{\xi}^{\beta,m \beta/2}(z)$ and (formally) exchange the infinite series and expectation. The computation then boils down to joint moments of the Taylor coefficients of $\boldsymbol{\xi}^{\beta,m \beta/2}$. These can in principle be computed using the Brownian motion representation from \cite{LiValko} or, at least for $\beta=2$, using exchangeability theory formulas or connections to Painlev\'{e} equations as developed in \cite{AGKW}. Performing this computation systematically is an interesting but formidable task. 

We also note that certain properties of the correlation functions, namely that they are uniformly bounded and behave like $\prod_{i<j}|x_i-x_j|^\beta$, up to a well-behaved (continuous and non-vanishing) multiplicative factor $\mathfrak{G}_{\beta;m}$ defined by \eqref{CorrFnFn}, when variables are close, have significant consequences in studying infinite-dimensional log-interacting diffusions using Dirichlet form theory, see \cite{Osada1,Osada2}. The exact form of $\mathfrak{G}_{\beta;m}$, which for generic $\beta$, is unlikely to have a completely explicit evaluation is not important for such considerations.

Finally, by comparing with the formula of Qu and Valk\'{o} from \cite{qu2025pair} for the pair correlation, which is the only other known expression for general $\beta>0$ (only for $m=2$), we obtain the following intriguing equality, for $x>0$,
\begin{equation}\label{ComparisonQuValko}
\mathbb{E}\left[\left|\boldsymbol{\xi}^{\beta,\beta}(x)\right|^{\beta}\right]=\frac{1}{\mathfrak{C}_\beta^{(2)}x^\beta}\left(\frac{1}{4\pi^2}+\frac{1}{2\pi^2}\sum_{k=1}^\infty \frac{\prod_{j=0}^{k-1}\left(-\frac{\beta}{2}+j\right)}{\prod_{j=0}^{k-1}\left(1+\frac{\beta}{2}+j\right)}\mathbb{E}\left[\cos \left(k\mathsf{p}_x^{(\beta,\beta/2)}(0)\right)\right]\right).
\end{equation}
Observe that, the left hand side of \eqref{ComparisonQuValko} involves  the $\mathcal{HP}_{\beta,\beta}$ point process while the right hand side involves the family (in the variable $x\in\mathbb{R}$) of diffusions $(\mathsf{p}_x^{(\beta,\beta/2)}(t))_{t\in (-\infty,\infty)}$ and thus $\mathcal{HP}_{\beta,\beta/2}$. It would be interesting to relate the two expressions directly.

Formula \eqref{CorrFormula} is very-well adapted to looking at the asymptotics when the variables $(x_1,\dots,x_m)$ merge. For example, the following corollary, which is immediate from Theorem \ref{MainThm2}, provides an answer to leading order for Problem 5 of \cite{qu2025pair} (in fact we consider the generalisation to all $m$-point correlation functions). It would be interesting to understand the lower order terms in $|x_i-x_j|$, at least for $m=2$.
\begin{cor}
For $\beta>0$, $m\in \mathbb{N}$, $x_*\in \mathbb{R}$
\begin{equation}
\lim_{(x_1,\dots,x_m)\to (x_*,\dots,x_*)}\frac{\boldsymbol{\rho}^{(m)}_\beta(x_1,\dots,x_m)}{\prod_{1 \le i <j \le m} \left|x_i-x_j \right|^\beta}=\mathfrak{C}_\beta^{(m)}.
\end{equation}
\end{cor}

Finally, it is interesting to note that the answers to the two problems addressed in Theorem \ref{MainThm1} and Theorem \ref{MainThm2} are in fact directly connected in one special case. Namely, by comparing the two formulae for $k=2,s=\beta/2$ and $m=2$ respectively, when $\beta>1$ we have
\begin{equation*}
\lim_{N\to \infty}N^{1-2\beta}\mathsf{M}_N^{(\beta)}(2;\beta/2)=\frac{\pi \beta^{\beta}}{2^{\beta-1}}\int_{-\infty}^\infty \frac{1}{|x|^\beta}\boldsymbol{\rho}_\beta^{(2)}(0,x)\mathrm{d}x.
\end{equation*}

\subsection{On the proof}

Our starting point to prove Theorems \ref{MainThm1} and \ref{MainThm2} is to write in Propositions \ref{MoMrep} and \ref{CbetaECorr} respectively both the moments of moments and the correlation functions of $\textnormal{C}\beta\textnormal{E}_N$ in terms of two different types of moments of the normalised characteristic polynomial $\mathfrak{q}_N^{\beta,\delta}$ of the circular Jacobi beta ensemble $\mathcal{CJ}_{N,\beta,\delta}$, see Definitions \ref{CJbetaEdef} and \ref{DefNormalisedChar}. The expression we obtain for $\mathsf{M}_N^{(\beta)}(k;s)$ in Proposition \ref{MoMrep} is  non-obvious and very judiciously chosen. It is worth comparing it with the natural rewriting of the $\mathsf{M}_N^{(\beta)}(k;s)$ formula to deal with the GMC regime, see Remark \ref{FormulaComparison}. It is then known from the work of Li and Valk\'{o} \cite{LiValko}, see Proposition \ref{LiValkoProp}, and \cite{CNN,ValkoVirag,RandomAnalytic,LambertPaquetteAiry,LambertPaquetteBulk} for related models, that under a certain scaling, $\mathfrak{q}_N^{\beta,\delta}$ converges in distribution, in the topology on entire functions induced by uniform convergence on compact sets in $\mathbb{C}$, to a random entire function $\tilde{\boldsymbol{\xi}}^{\beta,\delta}$ which we subsequently show in Proposition \ref{PropPVRep} that is equal in distribution to $\boldsymbol{\xi}^{\beta,\delta}$ from Definition \ref{StochasticZeta}.

To then prove the desired convergence of moments and establish Theorems \ref{MainThm1} and \ref{MainThm2} we require very precise moment estimates which are uniform in both $N$ and the variable of the polynomial $\mathfrak{q}_N^{\beta,\delta}$. All of these estimates will be derived from the same main bound of Theorem \ref{MainBound} which is the main technical innovation of our paper.  In equivalent form, see in particular Proposition \ref{FinalBoundProp}, this uniform bound says the following about joint moments of $(|\mathsf{X}_N(\mathrm{e}^{\mathrm{i}\theta_j})|^{2s_j},j=1,\dots,\ell)$,
\begin{equation*}
\mathbb{E}_{\mathrm{C}\beta\mathrm{E}_N}\left[\prod_{j=1}^\ell \left|\mathsf{X}_N\left(\mathrm{e}^{\mathrm{i}\theta_j}\right)\right|^{2s_j}\right]\le C_{\beta;\ell;(s_j)_{j=1}^\ell} N^{\frac{2}{\beta}\sum_{j=1}^\ell s_j^2}\prod_{1\le j<m \le \ell} \min\left(N,\frac{1}{2\left|\sin\left(\frac{\theta_j-\theta_m}{2}\right)\right|}\right)^{\frac{4}{\beta}s_js_m} 
\end{equation*}
for a constant $C_{\beta;\ell;(s_j)_{j=1}^\ell}$ which is independent of both $N$ and the $(\theta_j)_{j=1}^\ell$. A testament to the fact that this bound is sharp is that it captures perfectly the integrability required to get the optimal range, for $k\in \mathbb{N}$, in Theorem \ref{MainThm1}, see in particular Proposition \ref{PropIntegerMoments} and Lemma \ref{IntegralFinite}. Moreover, this result specialises for $\beta=2$ to the main result of Fahs \cite{Fahs}, when applied to the CUE characteristic polynomial, which is proven using a very delicate asymptotic analysis of Toeplitz determinants with symbols with multiple merging Fisher-Hartwig singularities. For $\beta\neq 2$ such Toeplitz determinant structure is absent. Instead we make use of random orthogonal polynomials on the unit circle first introduced by Killip and Nenciu in \cite{KillipNenciu} in relation to $\mathrm{C}\beta\mathrm{E}_N$. Parts of the argument are guided, rather implicitly, by some intuition coming from branching structures used in the study of conjecture \eqref{MaxConjecture} on the maximum of $\mathsf{L}_N(\cdot)$, see for example \cite{ChhaibiMadauleNajnudel}. We believe this bound on joint moments is of independent interest and will have other applications in the future.

\section{Proofs of main results}

In this section we prove all our results except the main bound in Theorem \ref{MainBound} whose rather lengthy proof is deferred to Section \ref{SectionMainBound}.

\begin{defn}\label{CJbetaEdef}
Define, for $\beta,\delta>0$, the circular Jacobi beta ensemble $\mathcal{CJ}_{N,\beta,\delta}$ to be the probability measure on $\mathbb{T}^N$ given by
\begin{equation*}
\frac{1}{\mathcal{Z}_{N,\beta,\delta}}\prod_{1\le j <k \le N} \left|\mathrm{e}^{\textnormal{i}\theta_j}-\mathrm{e}^{\textnormal{i}\theta_k}\right|^\beta \prod_{j=1}^N \left|1-\mathrm{e}^{\textnormal{i}\theta_j}\right|^{2\delta}\mathrm{d}\theta_1\cdots \mathrm{d}\theta_N,
\end{equation*} 
where $\mathcal{Z}_{N,\beta,\delta}$ is an explicit normalisation constant.
\end{defn}
Observe that, when $\delta=0$, $\mathcal{CJ}_{N,\beta,0}$ is simply $\textnormal{C}\beta \textnormal{E}_N$. $\mathcal{CJ}_{N,\beta,\delta}$ can in fact be defined more generally for $\Re(\delta)>-1/2$ but we will not need this here.  The following notation will be very convenient.

\begin{defn}\label{DefNormalisedChar}
For $\beta,\delta>0$, $N \in \mathbb{N}$, define the random polynomial $\mathfrak{q}^{\beta,\delta}_N$ by
\begin{equation}
\mathfrak{q}^{\beta,\delta}_N(z)\overset{\textnormal{def}}{=}\prod_{j=1}^N \frac{z-\mathrm{e}^{\mathrm{i}\theta_j}}{1-\mathrm{e}^{\mathrm{i}\theta_j}}
\end{equation}
where the random points $(\mathrm{e}^{\mathrm{i}\theta_j})_{j=1}^N$  are assumed to be $\mathcal{CJ}_{N,\beta,\delta}$-distributed.  
\end{defn}

The following two exact finite $N$ results will be our starting points for Theorem \ref{MainThm1} and Theorem \ref{MainThm2} respectively. 

\begin{prop}\label{MoMrep}
For $\beta>0$, $k,s \in \mathbb{R}_+$, we have
\begin{equation*}
\mathsf{M}^{(\beta)}_N(k;s)=\mathbb{E}_{\textnormal{C}\beta \textnormal{E}_N}\left[\left|\mathsf{X}_N(1)\right|^{2ks}\right] \mathbb{E}\left[\left(\frac{1}{2\pi}\int_{-\pi}^\pi \left|\mathfrak{q}_N^{\beta,ks}(\mathrm{e}^{\mathrm{i}\theta})\right|^{2s}\mathrm{d}\theta\right)^{k-1}\right].
\end{equation*}
\end{prop}

\begin{proof}
We write, using rotation invariance of $\textnormal{C}\beta \textnormal{E}_N$
\begin{align*}
\mathsf{M}^{(\beta)}_N(k;s)&=\frac{1}{2\pi}\int_{-\pi}^\pi \mathbb{E}_{\textnormal{C}\beta \textnormal{E}_N}\left[\left|\mathsf{X}_N(\mathrm{e^{\mathrm{i}\phi}})\right|^{2s}\left(\frac{1}{2\pi}\int_{-\pi}^\pi \left|\mathsf{X}_N(\mathrm{e}^{\mathrm{i}\theta})\right|^{2s}\mathrm{d}\theta\right)^{k-1}\right]\mathrm{d}\phi\\
&=\mathbb{E}_{\textnormal{C}\beta \textnormal{E}_N}\left[\left|\mathsf{X}_N(1)\right|^{2ks}\left(\frac{1}{2\pi}\int_{-\pi}^\pi \frac{\left|\mathsf{X}_N(\mathrm{e}^{\mathrm{i}\theta})\right|^{2s}}{\left|\mathsf{X}_N(1)\right|^{2s}}\mathrm{d}\theta\right)^{k-1}\right].
\end{align*}
Then, a simple change of measure from $\textnormal{C}\beta \textnormal{E}_N$ to $\mathcal{CJ}_{N,\beta,ks}$ gives the result.
\end{proof}

\begin{rmk}\label{FormulaComparison}
It is worth comparing the formula for $\mathsf{M}_N^{(\beta)}(k;s)$ obtained in Proposition \ref{MoMrep} to the more obvious expression
\begin{equation*}
\mathsf{M}_N^{(\beta)}(k;s)= \left(\mathbb{E}_{\textnormal{C}\beta \textnormal{E}_N}\left[\left|\mathsf{X}_N(1)\right|^{2s}\right]  \right)^k \mathbb{E}_{\textnormal{C}\beta \textnormal{E}_N}\left[\left(\frac{1}{2\pi}\int_{-\pi}^\pi \frac{\left|\mathsf{X}_N(\mathrm{e}^{\mathrm{i}\theta})\right|^{2s}}{\mathbb{E}_{\textnormal{C}\beta \textnormal{E}_N}\left[\left|\mathsf{X}_N(\mathrm{e}^{\mathrm{i}\theta})\right|^{2s}\right] }\mathrm{d}\theta\right)^k\right].
\end{equation*}
From \eqref{GMCconvergence} one expects and indeed it can be proven for part of the subcritical regime, at least when $\beta=2$, see Appendix A of \cite{KeatingWong}, that the second expectation above converges to,
\begin{equation*}
\frac{1}{(2\pi)^k}\mathbb{E}\left[\mathsf{GMC}_{2s\beta^{-\frac{1}{2}}}(\mathbb{T})^k\right]=\frac{\Gamma\left(1-\frac{2ks^2}{\beta}\right)}{\Gamma \left(1-\frac{2s}{\beta}\right)^k},
\end{equation*}
with the explicit evaluation by virtue of the Fyodorov-Bouchaud formula \cite{FyodorovBouchaud,Remy,ChhaibiNajnudel}. Note that this formula does not make sense in the regime we are concerned with, $2ks^2>\beta$.
\end{rmk}

\begin{prop}\label{CbetaECorr}
Let $\boldsymbol{\rho}_{N,\beta}^{(m)}$ be the $m$-th correlation function of $\textnormal{C}\beta \textnormal{E}_N$. Then, for $N > m$, 
\begin{align*}
\boldsymbol{\rho}_{N,\beta}^{(m)}(x_1,\dots,x_m)&=\frac{N!}{(N-m)!}\frac{\mathcal{Z}_{N-m,\beta}}{\mathcal{Z}_{N,\beta}}  \prod_{1 \le j <k \le m}\left|\mathrm{e}^{\mathrm{i}x_j}-\mathrm{e}^{\mathrm{i}x_k}\right|^\beta \times  \\ 
& \ \ \ \  \mathbb{E}_{\textnormal{C}\beta \textnormal{E}_{N-m}}\left[\left|\mathsf{X}_{N-m}(1)\right|^{m\beta}\right] \mathbb{E}\left[\prod_{j=2}^m \left|\mathfrak{q}_{N-m}^{\beta,m\beta/2}\left(\mathrm{e}^{\mathrm{i}(x_j-x_1)}\right)\right|^\beta\right].
\end{align*}
\end{prop}
\begin{proof}
By definition, we have
\begin{align*}
\boldsymbol{\rho}_{N,\beta}^{(m)}(x_1,\dots,x_m)=\frac{N!}{(N-m)!} \frac{1}{\mathcal{Z}_{N,\beta}}&\int_{[-\pi,\pi)^{N-m}} \prod_{1\le j <k \le m}\left|\mathrm{e}^{\mathrm{i}x_j}-\mathrm{e}^{\mathrm{i}x_k}\right|^\beta \prod_{j=1}^m \prod_{k=m+1}^N\left|\mathrm{e}^{\mathrm{i}x_j}-\mathrm{e}^{\mathrm{i}\theta_k}\right|^\beta \\
&\prod_{m+1 \le j <k \le N}\left|\mathrm{e}^{\mathrm{i}\theta_j}-\mathrm{e}^{\mathrm{i}\theta_k}\right|^\beta\mathrm{d}\theta_{m+1}\cdots \mathrm{d}\theta_N.
\end{align*}
Hence, we can write,
\begin{equation*}
\boldsymbol{\rho}_{N,\beta}^{(m)}(x_1,\dots,x_m)=\frac{N!}{(N-m)!}\frac{\mathcal{Z}_{N-m,\beta}}{\mathcal{Z}_{N,\beta}}\prod_{1\le j<k \le m}\left|\mathrm{e}^{\mathrm{i}x_j}-\mathrm{e}^{\mathrm{ix_k}}\right|^\beta \mathbb{E}_{\textnormal{C}\beta \textnormal{E}_{N-m}}\left[\prod_{j=1}^m\left|\mathsf{X}_{N-m}(\mathrm{e}^{\mathrm{i}x_j})\right|^\beta\right].
\end{equation*}
Using rotational invariance of $\textnormal{C}\beta \textnormal{E}_N$, the expectation in the last display is equal to
\begin{equation*}
\mathbb{E}_{\textnormal{C}\beta \textnormal{E}_{N-m}}\left[\left|\mathsf{X}_{N-m}(1)\right|^\beta \prod_{j=2}^m \left|\mathsf{X}_{N-m}\left(\mathrm{e}^{\mathrm{i}(x_j-x_1)}\right)\right|^\beta\right]
\end{equation*}
and the result follows by a change of measure to $\mathcal{CJ}_{N-m,\beta,m\beta/2}$.
\end{proof}

We need some preliminary results from the literature. Firstly, the following lemma is well-known, see \cite{RandomDets}.
\begin{lem}\label{AsymptoticsPartition}
For $\beta>0$, $r \in \mathbb{R}_+$, we have, as $N \to \infty$,
\begin{equation*}
N^{-\frac{2r^2}{\beta}}\mathbb{E}_{\textnormal{C}\beta \textnormal{E}_N}\left[\left|\mathsf{X}_N(1)\right|^{2r}\right]\longrightarrow \mathrm{e}^{\mathcal{Y}_\beta\left(1+2r-\beta/2\right)-2\mathcal{Y}_\beta \left(1+r-\beta/2\right)+\mathcal{Y}_\beta(1-\beta/2)}.
\end{equation*}
\end{lem}

We also recall the following result due to Li-Valk\'{o} \cite{LiValko}. It generalises previous results of Valk\'{o}-Virag \cite{ValkoVirag} to $\delta>0$. For the special case $\beta=2,\delta=0$ see even earlier work \cite{CNN}.

\begin{prop}\label{LiValkoProp} Let $\beta,\delta>0$.
Then, there exists a coupling of the $(\mathcal{CJ}_{N,\beta,\delta})_{N=1}^\infty$ and $\mathcal{HP_{\beta,\delta}}$ such that almost surely, as $N \to \infty$,
\begin{equation*}
\mathfrak{q}^{\beta,\delta}_N(\mathrm{e}^{\mathrm{i}z/N})\mathrm{e}^{-\mathrm{i}z/2}\longrightarrow\tilde{\boldsymbol{\xi}}^{\beta,\delta}(z),
\end{equation*}
uniformly on compacts in $z\in \mathbb{C}$ and the entire function's $\tilde{\boldsymbol{\xi}}_{\infty}^{\beta,\delta}$ zero set is exactly given by $\mathcal{HP}_{\beta,\delta}$.
\end{prop}

Li-Valk\'{o} give an interesting characterisation of $\tilde{\boldsymbol{\xi}}_\infty^{\beta,\delta}$ in terms of the solution of an entire function-valued SDE and also a representation of the Taylor coefficients of $\tilde{\boldsymbol{\xi}}_\infty^{\beta,\delta}$ in terms of iterated integrals of Brownian motions, see \cite{LiValko}. To connect to the principal value product from Definition \ref{StochasticZeta} we prove the following.  

\begin{prop}\label{PropPVRep}
For $\beta, \delta >0$, we have 
\begin{equation}
\boldsymbol{\xi}^{\beta,\delta}\overset{\textnormal{d}}{=} \tilde{\boldsymbol{\xi}}^{\beta,\delta}.
\end{equation}
Moreover, convergence in the product representation \eqref{PVrep} of $\boldsymbol{\xi}^{\beta,\delta}$ is uniform on compact sets in $\mathbb{C}$.
\end{prop}

\begin{proof} 
Firstly, we claim that almost surely $\tilde{\boldsymbol{\xi}}_\infty^{\beta,\delta}$ belongs to the Cartwright class of entire functions $\mathsf{CC}$ defined as
\begin{equation*}
\mathsf{CC}=\left\{f  \ \textnormal{entire}: \exists \Delta >1 \textnormal{ so that } |f(z)|\le  \Delta^{1+|z|}, \textnormal{ for all } z \in \mathbb{C}, \textnormal{ and }  \int_{-\infty}^\infty\frac{\log_+|f(x)|}{1+x^2} \mathrm{d}x <\infty \right\}.
\end{equation*}
This is proven in Proposition \ref{CartwrightProp} below. Then, the argument follows the proof of Proposition 34 of \cite{ValkoVirag} which treats the case $\delta=0$. Namely, by Theorem 11 of Section V.4.4 of \cite{Levin} for $f \in \mathsf{CC}$, we have
\begin{equation*}
f(z)=cz^m\mathrm{e}^{\mathrm{i}bz}\lim_{R \to \infty} \prod_{|\lambda_k|<R}\left(1-\frac{z}{\lambda_k}\right)
\end{equation*}
where $b,c\in \mathbb{R}$, $m\in \mathbb{N}\cup\{0\}$ and $(\lambda_k)_{k \in \mathbb{Z}}$ are the non-zero roots of $f$.  We now apply this to $\tilde{\boldsymbol{\xi}}^{\beta,\delta}$. It is easy to see that, since $\tilde{\boldsymbol{\xi}}^{\beta,\delta}(0)=1$, and $\tilde{\boldsymbol{\xi}}^{\beta,\delta}$ maps the real line to the real line, see \cite{LiValko}, and has real roots, we must have $m,b=0$ and $c=1$, which gives the proof for pointwise equality. Now, by construction of $\tilde{\boldsymbol{\xi}}^{\beta,\delta}$ as the secular function of the random Dirac operator associated to the $\mathcal{HP}_{\beta,\delta}$ point process, see \cite{LiValko}, we have
\begin{equation*}
\tilde{\boldsymbol{\xi}}^{\beta,\delta}(z)=\mathrm{e}^{-\boldsymbol{\gamma}z}\prod_{\mathsf{x}\in \mathcal{HP}_{\beta,\delta}}\left(1-\frac{z}{\mathsf{x}}\right)\mathrm{e}^{z/\mathsf{x}},
\end{equation*}
the convergence being uniform on compact sets $z\in \mathbb{C}$, for some random variable $\boldsymbol{\gamma}$ ($\boldsymbol{\gamma}$ has a meaning in terms of random Dirac operators, see \cite{ValkoVirag,LiValko}, but we will not need this). We now determine $\boldsymbol{\gamma}$. Observe that we have, uniformly on compact sets in $z\in \mathbb{C}$,
\begin{equation*}
\tilde{\boldsymbol{\xi}}^{\beta,\delta}(z)=\lim_{R \to \infty}\exp\left(z\left(-\boldsymbol{\gamma}+\sum_{\substack{\mathsf{x} \in \mathcal{HP}_{\beta,\delta} \\ |\mathsf{x}|<R}}1/\mathsf{x}\right)\right)\prod_{\substack{\mathsf{x} \in \mathcal{HP}_{\beta,\delta} \\ |\mathsf{x}|<R}}\left(1-\frac{z}{\mathsf{x}}\right).
\end{equation*}
By taking $z \in \mathbb{C}$ with $z\neq 0$ and $\tilde{\boldsymbol{\xi}}^{\beta,\delta}(z)\neq 0$, we obtain (using the pointwise equality and the fact that the roots of 
$\tilde{\boldsymbol{\xi}}^{\beta,\delta}$ are given by the point process $ \mathcal{HP}_{\beta,\delta}$) that almost surely
\begin{equation*}
\boldsymbol{\gamma}=\lim_{R\to \infty}\sum_{\substack{\mathsf{x} \in \mathcal{HP}_{\beta,\delta} \\ |\mathsf{x}|<R}}\frac{1}{\mathsf{x}}.
\end{equation*}
Putting everything together completes the proof.
\end{proof}

\begin{rmk}
Observe that, the proof of Proposition \ref{PropPVRep} above in fact shows almost sure equality between $\tilde{\boldsymbol{\xi}}^{\beta,\delta}$ and $\boldsymbol{\xi}^{\beta,\delta}$ in the coupling afforded by Proposition \ref{LiValkoProp}.
\end{rmk}

\begin{prop}\label{CartwrightProp}
For any $\beta, \delta>0$, almost surely
\begin{equation*}
\tilde{\boldsymbol{\xi}}^{\beta,\delta}\in \mathsf{CC}.
\end{equation*}
\end{prop}
We will prove this proposition later in this section.

We now arrive to the following sharp bound which is the main technical contribution of our paper. It will be proven in Section \ref{SectionMainBound} using random orthogonal polynomials on the unit circle.
\begin{thm}\label{MainBound}
Let $\beta,\delta>0$, $\ell \in \mathbb{N}$ and $(x_1,\dots,x_\ell) \in \mathbb{R}^\ell$. Suppose $r_1,\dots,r_\ell \in \mathbb{R}_+$ are such that $\sum_{j=1}^\ell r_j\le2\delta$. Then,
\begin{equation}
\mathbb{E}\left[\left|\mathfrak{q}_N^{\beta,\delta}(\mathrm{e}^{\mathrm{i}x_1/N})\right|^{r_1}\cdots \left|\mathfrak{q}_N^{\beta,\delta}(\mathrm{e}^{\mathrm{i}x_\ell/N})\right|^{r_\ell}\right]\le C \prod_{j=1}^\ell \frac{1}{1+|x_j|^{(2\delta-\sum_{m=1}^\ell r_m)r_j/\beta}} \prod_{1\le i < j\le \ell} \frac{1}{1+|x_i-x_j|^{r_ir_j/\beta}},
\end{equation}
where the constant $C > 0$ depends only on 
$\beta$, $\delta$, $\ell$ and $(r_m)_{1 \leq m \leq \ell}$, but not on $N$ and $(x_m)_{1 \leq m \leq \ell}$. 
\end{thm}

The following corollary is immediate from Theorem \ref{MainBound} using Proposition \ref{LiValkoProp} and Fatou's lemma.

\begin{cor}\label{BoundLimitingEntire}
Let $\beta,\delta>0$. For $r_1,\dots,r_\ell \in \mathbb{R}_+$ such that $\sum_{j=1}^\ell r_j \le2\delta$ we have,
\begin{equation}
\mathbb{E}\left[\left|\tilde{\boldsymbol{\xi}}^{\beta,\delta}(x_1)\right|^{r_1}\cdots \left|\tilde{\boldsymbol{\xi}}^{\beta,\delta}(x_\ell)\right|^{r_\ell}\right] \le C \prod_{j=1}^\ell \frac{1}{1+|x_j|^{(2\delta-\sum_{m=1}^\ell r_m)r_j/\beta}} \prod_{1\le i < j\le \ell} \frac{1}{1+|x_i-x_j|^{r_ir_j/\beta}},
\end{equation}
 where the constant $C > 0$ depends only on 
$\beta$, $\delta$, $\ell$ and $(r_m)_{1 \leq m \leq \ell}$. 
\end{cor}

We now prove a number of intermediate results needed to prove Theorem \ref{MainThm2}.

\begin{lem}\label{NormCorrAsymptotics}
For $\beta>0$, $N, m\in \mathbb{N}$, with $N\ge m$, we have the following asymptotics as $N \to \infty$
\begin{align*}
 \frac{N!}{(N-m)!}\frac{\mathcal{Z}_{N-m,\beta}}{\mathcal{Z}_{N,\beta}}&\mathbb{E}_{\textnormal{C}\beta \textnormal{E}_{N-m}}\left[\left|\mathsf{X}_{N-m}(1)\right|^{m\beta}\right]\\&\sim \frac{2^{m\left(\frac{\beta}{2}-1\right)}}{\pi^m\beta^{\beta m/2}} \mathrm{e}^{\mathcal{Y}_\beta\left(1+\beta(m-\frac{1}{2})\right)-2\mathcal{Y}_\beta\left(1+\frac{\beta}{2}(m-1)\right)+\mathcal{Y}_\beta\left(1-\frac{\beta}{2}\right)} N^{\frac{\beta m}{2}(m-1)+m}.
\end{align*}
\end{lem}
\begin{proof}
This is direct computation using Lemma \ref{AsymptoticsPartition} and standard asymptotics for Gamma functions.
\end{proof}

\begin{prop}\label{PropAsymptCorr}
Let $\beta>0$, $m\in \mathbb{N}$, $(x_1,\dots,x_m)\in \mathbb{R}^m$. We have as $N \to \infty$,
\begin{align*}
&\mathbb{E}\left[\prod_{j=2}^m \left|\mathfrak{q}_{N-m}^{\beta,m\beta/2}\left(\mathrm{e}^{\mathrm{i}\frac{1}{N}(x_j-x_1)}\right)\right|^{\beta}\right] \prod_{1\le j <k \le m}\left|\mathrm{e}^{\mathrm{i}\frac{x_j}{N}}-\mathrm{e}^{\mathrm{i}\frac{x_k}{N}}\right|^\beta \\
&\sim N^{-\binom{m}{2}\beta}\mathbb{E}\left[\prod_{j=2}^m\left|\tilde{\boldsymbol{\xi}}^{\beta,m \beta/2}(x_j-x_1)\right|^\beta\right] \prod_{1 \le j <k \le m}\left|x_j-x_k\right|^\beta.
\end{align*}
\end{prop}
\begin{proof}
The asymptotics of the Vandermonde-type term are obvious. Thus, we only need to show that for any $(x_1,\dots,x_m)\in \mathbb{R}^m$, as $N \to \infty$,
\begin{equation*}
\mathbb{E}\left[\prod_{j=2}^m \left|\mathfrak{q}_{N-m}^{\beta,m\beta/2}\left(\mathrm{e}^{\mathrm{i}\frac{1}{N}(x_j-x_1)}\right)\right|^{\beta}\right] \longrightarrow \mathbb{E}\left[\prod_{j=2}^m\left|\tilde{\boldsymbol{\xi}}^{\beta,m \beta/2}(x_j-x_1)\right|^\beta\right].
\end{equation*}
For this we only need some uniform integrability, since by virtue of Proposition \ref{LiValkoProp} we have almost sure convergence (in the coupling of Proposition \ref{LiValkoProp}). It moreover suffices to have for some $r>\beta$, 
\begin{equation*}
\sup_{N\ge m+1}\mathbb{E}\left[\prod_{j=2}^m\left|\mathfrak{q}_{N-m}^{\beta,m\beta/2}\left(\mathrm{e}^{\mathrm{i}\frac{1}{N}(x_j-x_1)}\right)\right|^r\right]<\infty.
\end{equation*}
Picking $r$ so that $\beta<r\le\frac{m}{m-1}\beta$, this follows by virtue of the bound from Theorem \ref{MainBound}.
\end{proof}

\begin{prop}\label{LimitOfCorr}
Let $\beta>0$, $m \in \mathbb{N}$. For any continuous function $\mathsf{f}_m$ with compact support on $\mathbb{R}^m$, we have
\begin{align*}
&\int_{\mathbb{R}^m}\boldsymbol{\rho}_\beta^{(m)}(x_1,\dots,x_m)\mathsf{f}_m(x_1,\dots,x_m)\mathrm{d}x_1\cdots\mathrm{d}x_m\\&=\lim_{N\to \infty} \frac{1}{N^m}\int_{\mathbb{R}^m}\boldsymbol{\rho}_{N,\beta}^{(m)}\left(\frac{x_1}{N},\dots,\frac{x_m}{N}\right)\mathsf{f}_m(x_1,\dots,x_m)\mathrm{d}x_1\cdots\mathrm{d}x_m.
\end{align*}
\end{prop}

\begin{proof}
Let us write $\boldsymbol{\Xi}_{N,\beta}=\{\theta_1,\dots,\theta_N\}$ for the random point process with $(\theta_1,\dots,\theta_N)$  being $\textnormal{C}\beta \textnormal{E}_N$-distributed.  It is known that, see \cite{ValkoViragCarousel,KillipStoiciu}, that as random point processes, as $N \to \infty$,
\begin{equation}\label{PointProcessConv}
N\boldsymbol{\Xi}_{N,\beta}\overset{\mathrm{d}}\longrightarrow \mathsf{Sine}_\beta.
\end{equation}
We only need to show that the correlation functions $(\boldsymbol{\lambda}_{N,\beta}^{(m)})_{m=1}^\infty$  of $N\boldsymbol{\Xi}_{N,\beta}$ also converge to the corresponding ones for $\mathsf{Sine}_\beta$. First, by a simple change of variables we readily get $(\boldsymbol{\lambda}_{N,\beta}^{(m)})_{m=1}^\infty=(N^{-m}\boldsymbol{\rho}_{N,\beta}^{(m)}(x_1/N,\dots,x_m/N))_{m=1}^\infty$.
Now, for all $m\in \mathbb{N}$, making use of the formula from Proposition \ref{CbetaECorr}, the bound from Theorem \ref{MainBound}, along with the working in Lemma \ref{NormCorrAsymptotics} and Proposition \ref{PropAsymptCorr}, we obtain for some constant $C_{m,\beta}$:
\begin{equation}\label{UniformBoundCorr}
\sup_{N\ge m+1} \sup_{(x_1,\dots,x_m)\in \mathbb{R}^m} \boldsymbol{\lambda}_{N,\beta}^{(m)}(x_1,\dots,x_m) \le C_{m,\beta}.
\end{equation}
Hence, for all $\ell \in \mathbb{N}$ with $\mathscr{A}$ being an arbitrary compact Borel set we have, for universal constants $c_{\ell, m} > 0$, 
\begin{equation}\label{UniformBoundMomentPoints}
\mathbb{E}\left[\left( (N\boldsymbol{\Xi}_{N,\beta})\left(\mathscr{A}\right)\right)^\ell\right]=\sum_{m=1}^\ell c_{\ell,m} \int_{\mathscr{A}^m}\boldsymbol{\lambda}_{N,\beta}^{(m)}(x_1,\dots,x_m)\mathrm{d}x_1\cdots \mathrm{d}x_m\le C_{\mathscr{A},\ell,\beta},
\end{equation}
for some constant $C_{\mathscr{A},\ell,\beta} > 0$ independent of $N$. From \eqref{PointProcessConv}, we get, as $N \to \infty$,
\begin{equation}\label{PointProcessSumConv}
\sum_{\mathsf{y}_1\neq \mathsf{y}_2\cdots \neq \mathsf{y}_m \in N \boldsymbol{\Xi}_{N,\beta}}\mathsf{f}_m(\mathsf{y}_1,\dots,\mathsf{y}_m)\overset{\mathrm{d}}{\longrightarrow} \sum_{\mathsf{y}_1\neq \mathsf{y}_2\cdots \neq \mathsf{y}_m \in  \mathsf{Sine}_\beta}\mathsf{f}_m(\mathsf{y}_1,\dots,\mathsf{y}_m).
\end{equation}
Since by definition of correlation function \cite{BorodinDet,JohanssonDet}, we have
\begin{align*}
  \mathbb{E}\left[\sum_{\mathsf{y}_1\neq \mathsf{y}_2\cdots \neq \mathsf{y}_m \in N \boldsymbol{\Xi}_{N,\beta}}\mathsf{f}_m(\mathsf{y}_1,\dots,\mathsf{y}_m)\right] &=\frac{1}{N^m}\int_{\mathbb{R}^m}\boldsymbol{\rho}_{N,\beta}^{(m)}\left(\frac{x_1}{N},\dots,\frac{x_m}{N}\right)\mathsf{f}_m(x_1,\dots,x_m)\mathrm{d}x_1\cdots\mathrm{d}x_m,\\
  \mathbb{E}\left[\sum_{\mathsf{y}_1\neq \mathsf{y}_2\cdots \neq \mathsf{y}_m \in  \mathsf{Sine}_\beta}\mathsf{f}_m(\mathsf{y}_1,\dots,\mathsf{y}_m)\right]  &=\int_{\mathbb{R}^m}\boldsymbol{\rho}_\beta^{(m)}(x_1,\dots,x_m)\mathsf{f}_m(x_1,\dots,x_m)\mathrm{d}x_1\cdots\mathrm{d}x_m,
\end{align*}
the desired result follows using \eqref{PointProcessSumConv} by virtue of dominated convergence and uniform integrability, which is a consequence of the uniform bound \eqref{UniformBoundMomentPoints} since $\mathsf{f}_m$ is continuous with compact support. 
\end{proof}

We next prove a number of intermediate results used to establish Theorem \ref{MainThm1}

\begin{prop}\label{ConvDistrMoM}
Let $\beta,\delta,s>0$ and suppose $4s(\delta-s)>\beta$. Then, as $N \to \infty$,
\begin{equation}
\int_{-\pi N}^{\pi N} \left|\mathfrak{q}_N^{\beta,\delta}(\mathrm{e}^{\mathrm{i}x/N})\right|^{2s}\mathrm{d}x\overset{\mathrm{d}}{\longrightarrow}\int_{-\infty}^\infty \left|\tilde{\boldsymbol{\xi}}^{\beta,\delta}(x)\right|^{2s}\mathrm{d}x.
\end{equation}
\end{prop}

\begin{proof}
It is enough to prove that, with the coupling 
given in Proposition \ref{LiValkoProp}, 
$\mathsf{Y}_{N}$ tends to zero in probability, where
$$\mathsf{Y}_{N} \overset{\mathrm{def}}{=} \left| \int_{-\pi N}^{\pi N} \left|\mathfrak{q}_N^{\beta,\delta}(\mathrm{e}^{\mathrm{i}x/N})\right|^{2s}\mathrm{d}x
- \int_{-\infty}^\infty \left|\tilde{\boldsymbol{\xi}}^{\beta,\delta}(x)\right|^{2s}\mathrm{d}x \right|.
$$
For $0 < R \leq \pi N$, we have 
$$\mathsf{Y}_{N} \leq \mathsf{Y}_{N,R}
+  \mathsf{Z}_{N,R},$$
where 
\begin{align*}
\mathsf{Y}_{N,R}&\overset{\mathrm{def}}{=}\left|\int_{-R}^R \left|\mathfrak{q}_N^{\beta,\delta}(\mathrm{e}^{\mathrm{i}x/N})\right|^{2s}\mathrm{d}x-\int_{-R}^{R}\left|\tilde{\boldsymbol{\xi}}^{\beta,\delta}(x)\right|^{2s}\mathrm{d}x\right|,\\
\mathsf{Z}_{N,R}&\overset{\mathrm{def}}{=}\int_{\mathbb{R}\backslash[-R,R]} \left|\mathfrak{q}_N^{\beta,\delta}(\mathrm{e}^{\mathrm{i}x/N})\right|^{2s}\mathrm{d}x + \int_{\mathbb{R}\backslash[-R,R]} \left|\tilde{\boldsymbol{\xi}}^{\beta,\delta}(x)\right|^{2s}\mathrm{d}x. 
\end{align*}
Hence, for all $R, \epsilon > 0$, 
$$\underset{N \rightarrow \infty}{\limsup} 
\, \mathbb{P} (\mathsf{Y}_{N} \geq \epsilon) 
\leq \underset{N \rightarrow \infty}{\limsup} 
\, \mathbb{P} (\mathsf{Y}_{N,R} \geq \epsilon/2)  +  \underset{N \rightarrow \infty}{\limsup} 
\, \mathbb{P} (\mathsf{Z}_{N,R} \geq \epsilon/2).$$
By Proposition \ref{LiValkoProp}, 
$\mathsf{Y}_{N,R}$ almost surely converges to zero when $N \to \infty$, so the first upper limit in the right-hand side is equal to zero. 
By virtue of the bound from Theorem \ref{MainBound} and Corollary \ref{BoundLimitingEntire} we get, since $\delta > s$ by assumption, 
\begin{equation*}
 \limsup_{N \to \infty}\mathbb{E}\left[\mathsf{Z}_{N,R}\right]\le  C\int_{\mathbb{R}\backslash[-R,R]} \frac{\mathrm{d}x}{1+|x|^{4s(\delta-s)/\beta}},
\end{equation*}
for $C > 0$ depending only on $\beta, \delta$ and $s$. 
Hence, 
$$\underset{N \rightarrow \infty}{\limsup} 
\, \mathbb{P} (\mathsf{Y}_{N} \geq \epsilon)  \leq \underset{N \rightarrow \infty}{\limsup} 
\, \mathbb{P} (\mathsf{Z}_{N,R} \geq \epsilon/2)
\leq \frac{2 C}{\epsilon} \int_{\mathbb{R}\backslash[-R,R]} \frac{\mathrm{d}x}{1+|x|^{4s(\delta-s)/\beta}}
$$
for any $R > 0$. 
Since $4s (\delta - s) / \beta > 1$ by assumption, the right-hand side is finite and tends to zero when $R \to \infty$, which implies 
$$\underset{N \rightarrow \infty}{\limsup} 
\, \mathbb{P} (\mathsf{Y}_{N} \geq \epsilon)  =  0,$$
concluding the proof.
\end{proof}

\begin{prop}\label{PropIntegerMoments}
Let $\beta,\delta,s>0$ and $m \in \mathbb{N}$ be such that $2s(2\delta-(m+1)s)>\beta$ and $4m(\delta-m s)s>\beta$. Then, we have, as $N \to \infty$,
\begin{equation}
\mathbb{E}\left[\left(\int_{-\pi N}^{\pi N}\left|\mathfrak{q}_N^{\beta,\delta}(\mathrm{e}^{\mathrm{i}x/N})\right|^{2s}\mathrm{d}x\right)^m\right]\longrightarrow \mathbb{E}\left[\left(\int_{-\infty}^\infty \left|\tilde{\boldsymbol{\xi}}^{\beta,\delta}(x)\right|^{2s}\mathrm{d}x\right)^m\right].
\end{equation}
\end{prop}

\begin{proof}
Since $m\in \mathbb{N}$, we can expand the power and use Fubini-Tonelli theorem to bring the expectation inside the multiple integral. Define the functions:
\begin{align*}
\mathfrak{F}_{N,m}^{\beta,\delta}(x_1,\dots,x_m)&=\mathbb{E}\left[\left|\mathfrak{q}_N^{\beta,\delta}(\mathrm{e}^{\mathrm{i}x_1/N})\right|^{2s}\cdots \left|\mathfrak{q}_N^{\beta,\delta}(\mathrm{e}^{\mathrm{i}x_m/N})\right|^{2s}\right],\\
\mathfrak{F}_{\infty,m}^{\beta,\delta}(x_1,\dots,x_m)&=\mathbb{E}\left[\left|\tilde{\boldsymbol{\xi}}^{\beta,\delta}(x_1)\right|^{2s}\cdots \left|\tilde{\boldsymbol{\xi}}^{\beta,\delta}(x_m)\right|^{2s}\right].
\end{align*}
We make two claims from which the conclusion immediately follows by the dominated convergence theorem. The first claim is that for all $(x_1,\dots,x_m) \in \mathbb{R}^m$, as $N \to \infty$,
\begin{equation}\label{ConvergenceForDCT}
\mathfrak{F}_{N,m}^{\beta,\delta}(x_1,\dots,x_m) \longrightarrow \mathfrak{F}_{\infty,m}^{\beta,\delta}(x_1,\dots,x_m).
\end{equation}
The second claim is that for all $N \in \mathbb{N}$ and $(x_1,\dots,x_m) \in \mathbb{R}^m$,
\begin{equation}\label{BoundForDCT}
\mathfrak{F}_{N,m}^{\beta,\delta}(x_1,\dots,x_m) \le C \mathfrak{D}_m(x_1,\dots,x_m),
\end{equation}
for some constant $C$ independent of $N\in \mathbb{N}$ and $(x_1,\dots,x_m)\in \mathbb{R}^m$, where $\mathfrak{D}_m$ is given by:
\begin{equation*}
\mathfrak{D}_m(x_1,\dots,x_m)=\prod_{j=1}^{m}\frac{1}{1+|x_j|^{4(\delta-ms)s/\beta}} \prod_{1\le j < k \le m} \frac{1}{1+|x_j-x_k|^{4s^2/\beta}},
\end{equation*}
and satisfies
\begin{equation}\label{IntegrabilityForDCT}
\int_{\mathbb{R}^m}\mathfrak{D}_m(x_1,\dots,x_m)\mathrm{d}x_1\cdots\mathrm{d}x_m<\infty.
\end{equation}
To show \eqref{ConvergenceForDCT}, by virtue of the almost sure convergence afforded by Proposition \ref{LiValkoProp} (in the coupling therein) we only need some uniform integrability. In particular, it suffices to have that for some $r>s$,
\begin{equation*}
\sup_{N\ge 1} \mathbb{E}\left[\left|\mathfrak{q}_N^{\beta,\delta}(\mathrm{e}^{\mathrm{i}x_1/N})\right|^{2r}\cdots \left|\mathfrak{q}_N^{\beta,\delta}(\mathrm{e}^{\mathrm{i}x_m/N})\right|^{2r}\right]<\infty.
\end{equation*}
Picking $r$ such that $s<r\le\frac{\delta}{m}$ (since $ms<\delta$) this follows from Theorem \ref{MainBound}. For the second claim, the bound \eqref{BoundForDCT} is immediate from Theorem \ref{MainBound} while the integrability condition \eqref{IntegrabilityForDCT} follows from Lemma \ref{IntegralFinite} by taking $a=4(\delta-ms)s/\beta$ and $b=4s^2/\beta$ and observing that  $2s(2\delta-(m+1)s)>\beta$ and and $4m(\delta-m s)s>\beta$ are exactly equivalent to the conditions of Lemma \ref{IntegralFinite}. This completes the proof. 
\end{proof}

\begin{lem}\label{IntegralFinite}
Let $m\in \mathbb{N}$, $a,b \in \mathbb{R}_+$ such that $a+(m-1)b/2>1$ and $ma>1$. Then,
\begin{equation*}
\int_{\mathbb{R}^m}\prod_{j=1}^m\frac{1}{1+|x_j|^a}\prod_{1\le i< j \le m}\frac{1}{1+|x_i-x_j|^b}\mathrm{d}x_1\cdots\mathrm{d}x_m<\infty.
\end{equation*}
\end{lem}

\begin{proof}
The result is obvious for $m=1$ or $b=0$. Let us assume $m\ge 2$ and $b>0$. Since $|x_i - x_j| \geq \big||x_i| - |x_j| \big|$, the integral is, by symmetry, bounded by 
\begin{equation*}
2^m m! \int_{0\le x_1\le x_2\le\cdots\le x_m<\infty}\prod_{j=1}^m\frac{1}{(1+x_j)^a}\prod_{1\le i<j \le m}\frac{1}{(1+x_j-x_i)^b}\mathrm{d}x_1\cdots \mathrm{d}x_m.
\end{equation*}
It is then enough to check the finiteness of this last integral. 
By making the change of variables $y_1=x_1, y_2=x_2-x_1, \dots, y_m=x_m-x_{m-1}$, it is equal to
\begin{equation*}
\int_{\mathbb{R}_+^m} \prod_{j=1}^m \frac{1}{\left(1+\sum_{\ell=1}^jy_\ell\right)^a} \prod_{2\le i\le j\le m}\frac{1}{\left(1+\sum_{\ell=i}^j y_\ell\right)^b}\mathrm{d}y_1\cdots \mathrm{d}y_m.
\end{equation*}
Now, for $1\le i \le j \le m$, we have
$$ m \left(1+\sum_{\ell=i}^j y_{\ell} \right) \geq  \sum_{\ell = i}^j (1 + y_{\ell}).$$
Hence, it is enough to show
\begin{equation*}
 \int_{\mathbb{R}^{m}_+}\prod_{j=1}^m \frac{1}{\left(\sum_{\ell=1}^j (1 + y_\ell)\right)^a} \prod_{2\le i\le j\le m}\frac{1}{\left(\sum_{\ell=i}^j (1 + y_\ell) \right)^b}\mathrm{d}y_1\cdots \mathrm{d}y_m<\infty,
\end{equation*}
 or equivalently 
\begin{equation*}
 \int_{\mathbb{R}^{m}_{\ge 1}}\prod_{j=1}^m \frac{1}{\left(\sum_{\ell=1}^jy_\ell\right)^a} \prod_{2\le i\le j\le m}\frac{1}{\left(\sum_{\ell=i}^j y_\ell\right)^b}\mathrm{d}y_1\cdots \mathrm{d}y_m<\infty,
\end{equation*}
which is in turn implied by
\begin{equation*}
\int_{\mathbb{R}^{m}_{\ge 1}}\prod_{j=1}^m \frac{1}{\left(\max\{y_j,y_1\}\right)^a} \prod_{2\le i\le j\le m}\frac{1}{\left(\max\{y_i,y_j\}\right)^b}\mathrm{d}y_1\cdots \mathrm{d}y_m<\infty.
\end{equation*}
Write $\mathsf{Perm}[i;j]$ for the set of permutations of $\{i,\dots,j\}$.
Let us decompose $\mathbb{R}_{\ge 1}^m$ as follows
\begin{align*}
\mathbb{R}_{\ge 1}^m &=\bigcup_{1\le i \le m,  \sigma \in \mathsf{Perm}[2;m]}\mathsf{H}_{i;\sigma}, \\ \mathsf{H}_{i;\sigma}&=\left\{(y_1,y_2,\dots,y_m)\in \mathbb{R}^m:1\le y_{\sigma(2)}\le y_{\sigma(3)}\le \cdots \le y_{\sigma(i)} \le y_1  \le y_{\sigma(i+1)}\le \cdots \le y_{\sigma(m)}\right\},
\end{align*}
with the obvious convention for $i=1,m$, namely $y_1\le y_j$ and $y_1\ge y_j$ respectively for all $j=2,\dots,m$. Now, observe that on each $\mathsf{H}_{i;\sigma}$ we have:
\begin{equation*}
\prod_{j=1}^m\max\{y_j,y_1\}=y_1^{i}\prod_{j=i+1}^my_{\sigma(j)}, \ \ \prod_{2\le i \le j \le m} \max\{y_i,y_j\}=\prod_{j=2}^m y_{\sigma(j)}^{j-1}.
\end{equation*}
We can then bound the last integral by
\begin{equation*}
\sum_{1\le i \le m, \sigma \in \mathsf{Perm}[2;m]}\int_{\mathsf{H}_{i;\sigma}} y_1^{-ia}\prod_{j=2}^{i}y_{\sigma(j)}^{-(j-1)b}\prod_{j=i+1}^my_{\sigma(j)}^{-a-(j-1)b}\mathrm{d}y_1 \cdots \mathrm{d}y_m.
\end{equation*}
Observe that, for fixed $1\le i \le m$, the integrals over $\mathsf{H}_{i;\sigma}$ are the same for all $\sigma \in \mathsf{Perm}[2;m]$ and hence it suffices to check that each of the following $m$ integrals, where $1\le i \le m$, is indeed finite:
\begin{equation*}
\int_{1 \le x_1\le \cdots \le x_{i-1}\le t \le x_i \le \cdots \le x_{m-1}<\infty} t^{-ia}\prod_{j=1}^{i-1}x_j^{-jb}\prod_{j=i}^{m-1}x_j^{-a-jb}\mathsf{d}x_1\cdots \mathrm{d}x_{m-1}\mathrm{d}t<\infty.
\end{equation*}
By performing the integrations sequentially in $x_{m-1}, x_{m-2},\dots, x_i, t, x_{i-1}, \dots,x_1$ we require that each exponent on the $x_j$'s or $t$ is more negative than $-1$. A simple counting argument then gives the following constraints, for each $1\le i \le m $ (note that for $i=m$, constraint \eqref{Constraint1} is absent while for $i=1$, constraint \eqref{Constraint3} is absent):
\begin{align}
 -(m-j-1)+(m-j)a+b\left(\frac{m(m-1)}{2}-\frac{j(j-1)}{2}\right) -1 &>0, \ \ j=i,\dots,m-1,\label{Constraint1}\\
-(m-i)+ma+b\left(\frac{m(m-1)}{2}-\frac{i(i-1)}{2}\right)-1 &>0,\nonumber\\
  -(m-j)+ma+b\left(\frac{m(m-1)}{2}-\frac{j(j-1)}{2}\right)-1&>0, \ \ j=1,\dots,i-1.\label{Constraint3}
\end{align}
Observe that, for $b>0$, the functions in the variable $j$ defined by the left hand sides of \eqref{Constraint1} and \eqref{Constraint3} are concave and hence it suffices to only check the inequalities for $j=i,m-1$ and $j=1,i-1$ respectively.
Thus, overall we only need to check that the following constraints are satisfied:
\begin{align}
a+(m-1)b-1&>0,\label{Constraint1'}\\
-(m-i-1)+(m-i)a+b\left(\frac{m(m-1)}{2}-\frac{i(i-1)}{2}\right)-1&>0, \  \ i=1,\dots,m-1,\label{Constraint2'}\\
-(m-i)+ma+b\left(\frac{m(m-1)}{2}-\frac{i(i-1)}{2}\right)-1&>0, \ \ i=1,\dots, m, \label{Constraint3'}\\
-(m-i+1)+ma+b\left(\frac{m(m-1)}{2}-\frac{(i-1)(i-2)}{2}\right)-1&>0, \ \ i=2,\dots,m,\label{Constraint4'}\\
-(m-1)+ma+b\frac{m(m-1)}{2}-1&>0\label{Constraint5'}.
\end{align}
First, observe that \eqref{Constraint1'} is implied by \eqref{Constraint5'} which is in turn equivalent to $a+(m-1)b/2>1$. Then, observe again that, when $b>0$, the functions in the variable $i$ defined by the left-hand sides of \eqref{Constraint2'}, \eqref{Constraint3'} and \eqref{Constraint4'} are concave and  so it suffices to only check the constraints for $i=1,m-1$, $i=1,m$ and $i=2,m$ respectively. The resulting inequalities are given by,
\begin{align*}
 -(m-2)+(m-1)a+bm(m-1)/2-1&>0,\ \ \ \ \ \ \ \ \ \ \ \ \ \  \ \  a+b(m-1)-1>0, \\
 -(m-1)+ma+bm(m-1)/2-1&>0, \ \ \ \ \ \ \ \ \ \ \ \ \ \ \ \ \ \ \ \ \ \ \ \ \ \ \ \ \ \  \ ma-1>0,\\
 -(m-1)+ma+bm(m-1)/2-1&>0, \ \ \  -1+ma+b(m-1)-1>0.
\end{align*}
These inequalities can then be easily checked from $a+(m-1)b/2>1$ and $ma>1$ (also recalling $m\ge 2)$ which completes the proof.
\end{proof}

\begin{prop}\label{PropNonInteger}
Let $\beta,s>0$ and $k \in \mathbb{R}_{\ge 1}$ such that $2s^2(2k-\lceil k \rceil)>\beta$ and moreover for $k>1$, $4\lceil k-1\rceil(k-\lceil k-1\rceil)s^2>\beta$. Then, as $N \to \infty$,
\begin{equation*}
\mathbb{E}\left[\left(\int_{-\pi N}^{\pi N}\left|\mathfrak{q}_N^{\beta,ks}(\mathrm{e}^{\mathrm{i}x/N})\right|^{2s}\mathrm{d}x\right)^{k-1}\right]\longrightarrow \mathbb{E}\left[\left(\int_{-\infty}^\infty \left|\tilde{\boldsymbol{\xi}}^{\beta,ks}(x)\right|^{2s}\mathrm{d}x\right)^{k-1}\right].
\end{equation*}
\end{prop}

\begin{proof} The case $k=1$ is trivial. For $k\in \mathbb{N}$, with $k\ge 2$, the statement follows by taking $\delta=ks$ and $m=k-1$ in Proposition \ref{PropIntegerMoments}. Let us assume then that $k\in \mathbb{R}_{\ge 1} \backslash \mathbb{N}$. By taking $\delta=ks$ in Proposition \ref{ConvDistrMoM}, which is allowed since for $k>1$, $4(ks-s)s\ge 2s^2(2k-\lceil k \rceil)>\beta$, we get, as $N \to \infty$,
\begin{equation*}
\int_{-\pi N}^{\pi N}\left|\mathfrak{q}_N^{\beta,ks}(\mathrm{e}^{\mathrm{i}x/N})\right|^{2s}\mathrm{d}x\overset{\mathrm{d}}{\longrightarrow}\int_{-\infty}^{\infty}\left|\tilde{\boldsymbol{\xi}}^{\beta,ks}(x)\right|^{2s}\mathrm{d}x.
\end{equation*} 
Thus, to get the required statement we only need some uniform integrability which will be implied by convergence of the next larger integer moment. Let us take $\delta=ks$ and $m=\lceil k-1 \rceil$ in Proposition \ref{PropIntegerMoments}. This is valid since the two conditions $2s^2(2k-\lceil k \rceil)>\beta$ and $4\lceil k-1\rceil(k-\lceil k-1\rceil)s^2>\beta$ exactly imply the conditions $2s(2\delta-(m+1)s)>\beta$ and $4m(\delta-m s)s>\beta$ required in Proposition \ref{PropIntegerMoments}. This concludes the proof.
\end{proof}

We now return to the proof of Proposition \ref{CartwrightProp}.

\begin{proof}[Proof of Proposition \ref{CartwrightProp}] The integral condition on the real line follows immediately from Corollary \ref{BoundLimitingEntire}, since for $0 < s\le2\delta$,
\begin{equation*}
\mathbb{E}\left[\log_+ \left|\tilde{\boldsymbol{\xi}}^{\beta,\delta}(x)\right|\right] \leq c_s \mathbb{E}\left[\left|\tilde{\boldsymbol{\xi}}^{\beta,\delta}(x)\right|^{s}\right]\le \frac{C c_s}{1+|x|^{(2\delta-s)s/\beta}} \le C c_s\ \ \forall x\in \mathbb{R},
\end{equation*}
for $c_s > 0$ depending only on $s$. 
To get the exponential bound for all $z\in \mathbb{C}$, suppose we could show that for some $\mathsf{c},\tilde{\mathsf{c}}>0$, we have for all $k\in \mathbb{N}$,
\begin{equation}\label{ProbBound}
\mathbb{P}\left(\max_{|z|\le k}\left|\tilde{\boldsymbol{\xi}}^{\beta,\delta}(z)\right|\ge 2^{\mathsf{c}k}\right)\le 2^{-\tilde{\mathsf{c}}k}.
\end{equation}
Then, by the Borel-Cantelli lemma, the conclusion would follow. Let $s \in (0,2\delta]$. We have
\begin{align*}
\mathbb{P}\left(\max_{|z|\le k}\left|\tilde{\boldsymbol{\xi}}^{\beta,\delta}(z)\right|\ge 2^{\mathsf{c}k}\right)=\mathbb{P}\left(\max_{|z|\le k}\left|\tilde{\boldsymbol{\xi}}^{\beta,\delta}(z)\right|^s\ge 2^{\mathsf{c}sk}\right)\le \frac{\mathbb{E}\left[\max_{|z|\le k} \left|\tilde{\boldsymbol{\xi}}^{\beta,\delta}(z)\right|^s\right]}{2^{\mathsf{c}sk}}.
\end{align*}
Observe that, since $\tilde{\boldsymbol{\xi}}^{\beta,\delta}$ is an entire function, for any $s>0$, $|\tilde{\boldsymbol{\xi}}^{\beta,\delta}|^s$ is subharmonic (this is a consequence of two well-known facts: (a) for holomorphic $z \mapsto h(z)$, $z\mapsto \log|h(z)|$ is subharmonic and (b) the composition of a subharmonic function $g$ with an increasing convex function $\phi$, $\phi\circ g$ is again subharmonic; thus taking $\phi(x)=\exp(sx)$, for any $s>0$, gives the claim). We now require two facts: for a subharmonic function $f$ on $\mathbb{C}$, we have the maximum principle $\max_{|z|\le k} f(z)=\max_{|z|=k}f(z)$ and moreover, if $f$ is non-negative, we have by Poisson's inequality, again for a generic constant $C>0$ used throughout this proof, which may change from line to line,
\begin{align*}
\max_{|z|=k} f(z)&\le \max_{|z|=k}\frac{1}{2\pi}\int_{0}^{2\pi}f\left((k+1)\mathrm{e}^{\mathrm{i}\theta}\right)\frac{(k+1)^2-k^2}{\left|(k+1)\mathrm{e}^{\mathrm{i}\theta}-z\right|^2}\mathrm{d}\theta \le C(k+1)\int_{|z|=k+1}f(z)|\mathrm{d}z|.
\end{align*}
Hence, applying this to $f(\cdot)=|\tilde{\boldsymbol{\xi}}^{\beta,\delta}(\cdot)|^s$ we obtain,
\begin{align*}
\mathbb{P}\left(\max_{|z|\le k}\left|\tilde{\boldsymbol{\xi}}^{\beta,\delta}(z)\right|\ge 2^{\mathsf{c}k}\right) \le\frac{C(k+1)}{2^{\mathsf{c}sk}}\mathbb{E}\left[\int_{|z|=k+1}\left|\tilde{\boldsymbol{\xi}}^{\beta,\delta}(z)\right|^{s}|\mathrm{d}z|\right]\le\frac{C(k+1)^2}{2^{\mathsf{c}sk}}\sup_{|z|=k+1}\mathbb{E}\left[\left|\tilde{\boldsymbol{\xi}}^{\beta,\delta}(z)\right|^s\right].
\end{align*}
For $z=x+\mathrm{i}y$ we have, by virtue of Proposition \ref{LiValkoProp} and Fatou's lemma,
\begin{equation}\label{BoundXiPlane}
\mathbb{E}\left[\left|\tilde{\boldsymbol{\xi}}^{\beta,\delta}(x+\mathrm{i}y)\right|^s\right] \le \mathrm{e}^{ys/2} \liminf_{N\to \infty}\mathbb{E}\left[\left|\mathfrak{q}_N^{\beta,\delta}(\mathrm{e}^{-y/N}\mathrm{e}^{\mathrm{i}x/N})\right|^s\right].
\end{equation}
For $r<1$, with $\mathcal{P}_r(\theta)=(2\pi)^{-1}\sum_{n=-\infty}^\infty r^{|n|}\mathrm{e}^{\mathrm{i}n\theta}$ the normalised Poisson kernel, we have,
\begin{equation}
\mathbb{E}\left[\left|\mathfrak{q}_N^{\beta,\delta}(r\mathrm{e}^{\mathrm{i}\theta})\right|^s\right]\le \int_{-\pi}^\pi \mathcal{P}_{r}(\theta-t) \mathbb{E}\left[\left|\mathfrak{q}_N^{\beta,\delta}(\mathrm{e}^{\mathrm{i}t})\right|^s\right]\mathrm{d}t\le C,
\end{equation}
the final inequality due to the bound in Theorem \ref{MainBound}.
Now, by virtue of rotational invariance of $\textnormal{C}\beta \textnormal{E}_N$ we have, for $r>1$, the functional equation
\begin{equation*}
\left|\mathfrak{q}_N^{\beta,\delta}(r\mathrm{e}^{\mathrm{i}\theta})\right|^s\overset{\mathrm{d}}{=}r^{Ns}\left|\mathfrak{q}_N^{\beta,\delta}(r^{-1}\mathrm{e}^{-\mathrm{i}\theta})\right|^s,
\end{equation*}
and so we get the following bound on the whole complex plane:
\begin{equation*}
  \mathbb{E}\left[\left|\mathfrak{q}_N^{\beta,\delta}(r\mathrm{e}^{\mathrm{i}\theta})\right|^s\right]  \le \begin{cases}
        C, \ \ &r \le 1, \theta \in \mathbb{T},\\
        C r^{Ns}, \ \ &r> 1, \theta \in \mathbb{T}.
    \end{cases}
\end{equation*}
Thus, by plugging in \eqref{BoundXiPlane} we obtain, 
\begin{equation*}
\mathbb{E}\left[\left|\tilde{\boldsymbol{\xi}}^{\beta,\delta}(z)\right|^s\right] \le C\mathrm{e}^{s|z|/2}, \ \ \ \forall z \in \mathbb{C}.
\end{equation*}
Putting everything together we get, 
\begin{equation*}
\mathbb{P}\left(\max_{|z|\le k}\left|\tilde{\boldsymbol{\xi}}^{\beta,\delta}(z)\right|\ge 2^{\mathsf{c}k}\right)\le C(k+1)^22^{sk(1/(2\log 2)-\mathsf{{c}})}.
\end{equation*}
Hence, picking $\mathsf{c}>1/(2 \log 2)$ (and appropriately $\tilde{\mathsf{c}}>0$ small enough) gives \eqref{ProbBound} and completes the proof.

\end{proof}

\begin{proof} [Proof of Theorem \ref{MainThm1}]
This follows immediately by computing the asymptotics of the formula in Proposition \ref{MoMrep} by virtue of Lemma \ref{AsymptoticsPartition} and Proposition \ref{PropNonInteger} (note that for $k\in \mathbb{N}$, $k\ge 2$, both conditions are implied by $2ks^2>\beta$) and finally recalling Proposition \ref{PropPVRep}.
\end{proof}

\begin{proof} [Proof of Theorem \ref{MainThm2}]
  We combine the expression given in Proposition \ref{CbetaECorr} with Lemma \ref{NormCorrAsymptotics} and Proposition \ref{PropAsymptCorr}, and we use the uniform bound \eqref{UniformBoundCorr} in order to apply 
  dominated convergence. In this way, we obtain 
the value of the limit
\begin{equation*}
\lim_{N\to \infty} \frac{1}{N^m}\int_{\mathbb{R}^m}\boldsymbol{\rho}_{N,\beta}^{(m)}\left(\frac{x_1}{N},\dots,\frac{x_m}{N}\right)\mathsf{f}_m(x_1,\dots,x_m)\mathrm{d}x_1\cdots\mathrm{d}x_m,
\end{equation*}
for any continuous function $\mathsf{f}_m$ of compact support on $\mathbb{R}^m$.
Combining the explicit expression of this limit with 
 Proposition \ref{PropPVRep} and  Proposition \ref{LimitOfCorr}, we deduce 
\begin{align*}
&\int_{\mathbb{R}^m}\boldsymbol{\rho}_\beta^{(m)}(x_1,\dots,x_m)\mathsf{f}_m(x_1,\dots,x_m)\mathrm{d}x_1\cdots\mathrm{d}x_m\\
&=\mathfrak{C}_{\beta}^{(m)} \int_{\mathbb{R}^m}\prod_{1\le i<j \le m}\left|x_i-x_j\right|^\beta \mathbb{E}\left[\prod_{j=2}^m \left|\boldsymbol{\xi}^{\beta,m\beta/2}\left(x_j-x_1\right)\right|^\beta\right]\mathsf{f}_m(x_1,\dots,x_m)\mathrm{d}x_1\cdots\mathrm{d}x_m.
\end{align*}
Since such $\mathsf{f}_m$ form a determining class we obtain expression \eqref{CorrFormula}. Then, uniform boundedness follows immediately by plugging into \eqref{CorrFormula} the bound from Corollary \ref{BoundLimitingEntire}.

To show continuity of \eqref{CorrFnFn} we only need some uniform integrability to get convergence of moments (since almost surely $z\mapsto \boldsymbol{\xi}^{\beta,\delta}(z)$ is continuous, in fact entire). It suffices to have, with some $r>\beta$, for all $R>0$,
\begin{align*}
\sup_{(x_1,\dots,x_{m-1})\in [-R,R]^{m-1}}\mathbb{E}\left[\prod_{j=1}^{m-1}\left|\boldsymbol{\xi}^{\beta,m\beta/2}(x_j)\right|^r\right]<\infty.
\end{align*}
Picking $r$ such that $\beta<r\le\frac{m}{m-1}\beta$, this follows from Corollary \ref{BoundLimitingEntire} and Proposition \ref{PropPVRep}.

We finally show that \eqref{CorrFnFn} is strictly positive for all $(x_1,\dots,x_m)\in \mathbb{R}^m$. Suppose otherwise. Then, by virtue of representation \eqref{PVrep} of $\boldsymbol{\xi}^{\beta,m\beta/2}$ there must exist a deterministic $y\in \mathbb{R}$ such that with positive probability $q$, the point process $\mathcal{HP}_{\beta,m\beta/2}$ has a point at $y$. Hence, for all $\epsilon >0$,
\begin{equation*}
q \le \mathbb{E}\left[\mathcal{HP}_{\beta,m\beta/2}\left[y-\epsilon,y+\epsilon\right]\right]=\int_{y-\epsilon}^{y+\epsilon}\boldsymbol{\rho}_{\beta,m\beta/2}^{(1)}(x)\mathrm{d}x \le 2\epsilon \sup_{x\in [y-1,y+1] } \boldsymbol{\rho}_{\beta,m\beta/2}^{(1)}(x),
\end{equation*}
where $\boldsymbol{\rho}_{\beta,m\beta/2}^{(1)}$ is the first correlation function of $\mathcal{HP}_{\beta,m\beta/2}$ which is known to be continuous (we only need that is bounded), see \cite{qu2025pair}, and this gives a contradiction by taking $\epsilon$ small enough.
\end{proof}

\section{Proof of the main bound}\label{SectionMainBound}
We will need several preliminaries. We consider the 
family $(\Phi_k, \Phi^*_k)_{k \geq 0}$, 
of orthogonal polynomials on the unit circle $\mathbb{T}$
corresponding to the sequence $(\boldsymbol{\alpha}_j)_{j \geq 0}$ of  random Verblunsky coefficients: these coefficients being independent, invariant in distribution by multiplication by complex numbers of modulus one, $|\boldsymbol{\alpha}_j|^2$ being 
distributed as a beta random variable with parameters $1$ and $(\beta/2)(j+1)$, namely with density for $x\in[0,1]$ given by,
\begin{equation*}
\frac{\beta(j+1)}{2}(1-x)^{\frac{\beta}{2}(j+1)-1}.
\end{equation*}
We have $\Phi_0 = \Phi^*_0 = 1$ and the Szeg\"o recursion, for $k\ge 0$: 
$$\Phi_{k+1} (z) = z \Phi_k(z) - \overline{\boldsymbol{\alpha}_k} \Phi^*_k(z),$$
$$\Phi^*_{k+1} (z) = \Phi^*_k (z) - z \boldsymbol{\alpha}_k \Phi_k(z). $$
Killip and Nenciu \cite{KillipNenciu}
 have proven that for all $N \geq 1$, 
$\theta\mapsto \mathsf{X}_N(\mathrm{e}^{\mathrm{i}\theta})$ has the same distribution as 
$$\theta \mapsto \Phi^*_{N-1} (\mathrm{e}^{\textnormal{i}\theta })- \mathrm{e}^{\textnormal{i}\theta } \boldsymbol{\eta} \Phi_{N-1} (\mathrm{e}^{\textnormal{i}\theta })$$
where $\boldsymbol{\eta}$ is an independent, uniform random variable on $\mathbb{T}$.
Moreover, for example by Lemma 2.3 of \cite{ChhaibiMadauleNajnudel},  we have for all 
$k \geq 0$, 
$$\Phi_k^*(\mathrm{e}^{\textnormal{i} \theta}) 
= \prod_{j=0}^{k-1} (1 -\boldsymbol{\alpha}_j \mathrm{e}^{\textnormal{i}\Psi_j(\theta)}),$$
where 
$$\Psi_{j} (\theta) = (j+1) \theta  -  
2 \sum_{r = 0}^{j-1} \Im \log \left(1 - \boldsymbol{\alpha}_r \mathrm{e}^{\textnormal{i} \Psi_r(\theta)}\right),$$
taking the principal branch of the logarithm in the last expression.

For $j \geq 0$, we denote by $\mathcal{F}_j$ the $\sigma$-algebra generated by the $(\boldsymbol{\alpha}_r)_{0 \leq r \leq j-1}$.

We have the following bound on fluctations of the variations of $\Psi_j$: 
\begin{prop} \label{boundlargefluctuationpsi}
For $j, \ell \geq 0$, $\theta \in \mathbb{R}$, $t > 0$, 
$$\mathbb{P} \left(| \Psi_{j+\ell} (\theta) - \Psi_j (\theta) - \ell \theta | \geq t | \mathcal{F}_j \right)
\leq 2 \exp \left( - \frac{t^2 \beta}{8 \log \left(1 + \frac{\beta \ell}{1 + \beta j} \right)}\right).$$
\end{prop}
\begin{proof}
By rotational invariance and independence of the Verblunsky coefficients, 
$$\Psi_{j+\ell}(\theta) - \Psi_j (\theta) - \ell \theta = - 2 \sum_{r = j}^{j+ \ell - 1} 
\Im \log ( 1 - \boldsymbol{\alpha}_r) $$
in distribution, and is independent of 
$\mathcal{F}_j$. 
By applying Proposition 2.5 of \cite{ChhaibiMadauleNajnudel}, we get, for all $\lambda \in \mathbb{R}$, 
$$\mathbb{E} \left[ \mathrm{e}^{\lambda (\Psi_{j+\ell}(\theta) - \Psi_j (\theta) - \ell \theta)} \big| \mathcal{F}_j\right] 
\leq \exp \left( 2 \lambda^2 \sum_{r = j}^{j+\ell - 1} \frac{1}{1 + \beta(r+1)}\right).
$$
Adding these estimates for $\lambda$ and $- \lambda$ gives, for all $\lambda \geq 0$, 
\begin{align*} \mathbb{E} \left[ \mathrm{e}^{\lambda |\Psi_{j+\ell}(\theta) - \Psi_j (\theta) - \ell \theta)|} \big| \mathcal{F}_j\right] & 
\leq 2 \exp \left( 2 \lambda^2 \sum_{r = j}^{j+\ell - 1} \frac{1}{1 + \beta(r+1)}\right)
\\ & \leq 2 \exp \left( 2 \lambda^2 \int_{j-1}^{j+\ell - 1} \frac{\mathrm{d}t}{1 + \beta(t+1)} \right)
\\ & = 2 \exp \left( 2 \lambda^2 \beta^{-1} \log \left( 
\frac{1 + \beta(j+\ell)}{1 + \beta j} \right) \right).
\end{align*}
Using the Chernoff bound, we get 
$$\mathbb{P} \left(| \Psi_{j+\ell} (\theta) - \Psi_j (\theta) - \ell \theta | \geq t | \mathcal{F}_j \right)
\leq 2 \exp \left(- \lambda t +  2 \lambda^2 \beta^{-1}\log \left( 1 + 
\frac{ \beta \ell}{1 + \beta j} \right) \right).$$
Finally, taking 
$$\lambda =\frac{ \beta t}{4  \log \left( 1 + 
\frac{ \beta \ell}{1 + \beta j} \right) } $$
gives the desired result. 

\end{proof}

Now, we obtain a bound for the conditional on $\mathcal{F}_j$, with $j\ge \ell$, moment of the product $|\Phi_{j+\ell}^*(\mathrm{e}^{\mathrm{i}\theta_1})|^{2s_1}\cdots |\Phi_{j+\ell}^*(\mathrm{e}^{\mathrm{i}\theta_p})|^{2s_p}$. This will be used as seed for iteration in Proposition \ref{boundconditionalmomentsdoubling}.
\begin{prop} \label{boundconditionalmoment}
For $p \geq 2$, let $(\theta_m)_{1 \leq m \leq p}\in \mathbb{R}^p$, $(s_m)_{1 \leq m \leq p}\in\mathbb{R}_+^p$, 
and $j,\ell \in \mathbb{N}\cup \{0\}$ so that $j \geq \ell$.  
Then, 
\begin{align*}
& \mathbb{E} \left[ \prod_{m=1}^p \left|\Phi^*_{j+\ell} (\mathrm{e}^{\textnormal{i} \theta_m})
\right|^{2s_m} \;  \bigg| \mathcal{F}_j \right]
\\ & \leq \prod_{m=1}^p \left|\Phi^*_{j} (\mathrm{e}^{\textnormal{i} \theta_m})
\right|^{2s_m} \, \exp \left[ 
\frac{2 \ell}{\beta (j+1)}  
 \sum_{m=1}^{p} s_m^2 
+ \frac{K}{j+1} \min \left( \ell, \left(\min_{1 \leq m_1 \neq m_2 \leq p} 
||\theta_{m_1} - \theta_{m_2}|| \right)^{-1} \right)  \dots \right.
\\ & \left.  \dots
+ K \left(\frac{\ell}{j+1} \right)^{3/2}    \left(1 + \sqrt{\log \left(\frac{1+j}{1 + \ell} \right) } \right)   
\right],
\end{align*}
where $K > 0$ depends only on $\beta, p$
and on the sequence $(s_m)_{1 \leq m \leq p}$, and $||a||$ denotes the distance between $a$ and the set $2 \pi \mathbb{Z}$. 
\end{prop}
\begin{proof}
We have, for 
$\ell \geq 1$, 
$$\mathbb{E} \left[ \prod_{m=1}^p \left|\Phi^*_{j+\ell} (\mathrm{e}^{\textnormal{i} \theta_m})
\right|^{2s_m} \;  \bigg| \mathcal{F}_{j+ \ell - 1} \right]
= \prod_{m=1}^p \left|\Phi^*_{j+\ell-1} (\mathrm{e}^{\textnormal{i} \theta_m})
\right|^{2s_m} \; 
\mathbb{E} \left[ \prod_{m=1}^p \left| 1 - \boldsymbol{\alpha}_{j+\ell-1} \mathrm{e}^{\textnormal{i} \Psi_{j+ \ell - 1} (\theta_m) }\right|^{2s_m} \bigg| \mathcal{F}_{j+ \ell-1} \right].
$$
If we take the principal branch of the powers
and use Newton's expansion, we get, for any $\rho \in (0,1)$,  
$$
\prod_{m=1}^p ( 1 - \rho \boldsymbol{\alpha}_{j+\ell-1} \mathrm{e}^{\textnormal{i} \Psi_{j+ \ell - 1} (\theta_m) })^{s_m} 
= \sum_{k=0}^{\infty} C_k \rho^k \boldsymbol{\alpha}^k_{j+\ell-1}
$$
where 
$$C_k = \sum_{(r_m)_{1 \leq m \leq p} \in
\{0,1,\dots, k\}^p, r_1 + \dots + r_p = k}
\prod_{m=1}^{p} (-\mathrm{e}^{\textnormal{i} \Psi_{j+ \ell - 1} (\theta_m) })^{r_m}
{s_m \choose r_m}.
$$
We have, for all $r \geq s_m \geq 0$, 
$$ \left|{s_m \choose {r+1}} \right|
=\left| {s_m \choose r} \right|\frac{r - s_m}{r+1} \leq \left| {s_m \choose r} \right|,$$
which shows that ${s_m \choose r} $ is bounded 
for fixed $s_m \geq 0$. 
We deduce that each term in the sum giving $C_k$ has a modulus bounded by a quantity $D > 0$ depending only on the sequence $(s_m)_{1 \leq m \leq p}$, and that 
$$|C_k| \leq D (k+1)^p. $$
Expanding the squared modulus gives: 
$$\prod_{m=1}^p \left| 1 - \rho \boldsymbol{\alpha}_{j+\ell-1} \mathrm{e}^{\textnormal{i} \Psi_{j+ \ell - 1} (\theta_m) }\right|^{2s_m} = 
\sum_{k_1, k_2 = 0}^{\infty}
C_{k_1} \overline{C_{k_2}} \rho^{k_1 + k_2}
\boldsymbol{\alpha}^{k_1}_{j+\ell-1} \overline{\boldsymbol{\alpha}_{j+\ell-1}^{k_2}},
$$
the double sum being absolutely convergent, uniformly in $\boldsymbol{\alpha}_{j+\ell-1}$ on the unit disc.  
From rotational invariance of the distribution of  $\boldsymbol{\alpha}_{j+\ell-1}$ and independence 
with $\mathcal{F}_{j+\ell-1}$, we get 
$$\mathbb{E} \left[ \prod_{m=1}^p \left| 1 - \rho \boldsymbol{\alpha}_{j+\ell-1} \mathrm{e}^{\textnormal{i} \Psi_{j+ \ell - 1} (\theta_m) }\right|^{2s_m} \, \bigg| \mathcal{F}_{j+ \ell-1} \right] 
= \sum_{k=0}^{\infty} 
|C_k|^2 \rho^{2k} \mathbb{E}[ |\boldsymbol{\alpha}_{j+\ell-1} |^{2k} ].$$
By dominated convergence for $\rho$ tending to $1$ from below, 
$$\mathbb{E} \left[ \prod_{m=1}^p \left| 1 -  \boldsymbol{\alpha}_{j+\ell-1} \mathrm{e}^{\textnormal{i} \Psi_{j+ \ell - 1} (\theta_m) }\right|^{2s_m} \, \bigg| \mathcal{F}_{j+ \ell-1} \right] 
= \sum_{k=0}^{\infty} 
|C_k|^2 \mathbb{E}[ |\boldsymbol{\alpha}_{j+\ell-1} |^{2k} ],$$
and then 
\begin{align*}\mathbb{E} \left[ \prod_{m=1}^p \left| 1 -  \boldsymbol{\alpha}_{j+\ell-1} \mathrm{e}^{\textnormal{i} \Psi_{j+ \ell - 1} (\theta_m) }\right|^{2s_m} \, \bigg| \mathcal{F}_{j+ \ell-1} \right] 
& = 1 + \frac{1}{ 1 + (\beta/2)(j+\ell)} \left|\sum_{m=1}^{p} s_m \mathrm{e}^{\textnormal{i} \Psi_{j+ \ell - 1} (\theta_m) }\right|^2
\\ & + \mathcal{O} \left( D^2 \sum_{k=2}^{\infty}
(k+1)^{2p} \mathbb{E}[ |\boldsymbol{\alpha}_{j+\ell-1} |^{2k}]
\right)
\end{align*}
since the expectation of 
$|\boldsymbol{\alpha}_{j+\ell-1}|^2$, which is beta-distributed with parameters equal to $1$ and $(\beta/2)(j+\ell)$, is equal to $1/ (1 + (\beta/2)(j+\ell))$. 
Moreover, 
$$\sum_{k=2}^{\infty} (k+1)^{2p}
 \mathbb{E}[ |\boldsymbol{\alpha}_{j+\ell-1} |^{2k}]
 = (\beta/2)(j+\ell) 
 \int_0^1 \left(\sum_{k = 2}^{\infty} x^k (k+1)^{2p}\right)  (1-x)^{(\beta/2)(j+\ell)  - 1}  \mathrm{d}x
$$
where 
$$(1-x)^{-2p-1}
= \sum_{k=0}^{\infty} x^k
\frac{(2p+1)(2p+2) \dots (2p+k)}{k!}
= \sum_{k=0}^{\infty} x^k
\frac{(k+1)(k+2) \dots (k+2p)}{(2p)!},
$$
$$x^2 (1-x)^{-2p-1} = 
\sum_{k=2}^{\infty} x^k
\frac{(k-1)k(k+1) \dots (k+2p-2)}{(2p)!}
\geq \frac{(1/3)(2/3)}{(2p)!}
\sum_{k=2}^{\infty} (k+1)^{2p} x^k
$$
and then
$$\sum_{k=2}^{\infty} (k+1)^{2p}
 \mathbb{E}[ |\boldsymbol{\alpha}_{j+\ell-1} |^{2k}]
 \leq (9 \beta/4) (2p)! (j+\ell) 
 \int_0^{1} x^2 (1 - x)^{(\beta/2)(j+\ell) - 2p-2} \mathrm{d}x.
 $$
If $j + \ell \geq (2/\beta)(2p+2)$, the last integral is finite, and 
we get 
\begin{align*}\sum_{k=2}^{\infty} (k+1)^{2p}
 \mathbb{E}[ |\boldsymbol{\alpha}_{j+\ell-1} |^{2k}]
 & \leq (9 \beta/4) (2p)! (j+\ell)  
 \frac{ \Gamma(3) \Gamma((\beta/2)(j+\ell) - 2p - 1)}{\Gamma((\beta/2)(j+\ell) - 2p +2 )}
 \\ & \leq \frac{ (9 \beta/2) (2p)! (j+\ell)}{
 ((\beta/2)(j+\ell) - 2p -1 )
 ((\beta/2)(j+\ell) - 2p  )((\beta/2)(j+\ell) - 2p +1 )},
 \end{align*}
which, for $j+\ell \geq (2/\beta)(4p + 2)$, 
is bounded, since $(\beta/2)(j+\ell) - 2p-1
\geq (\beta/4)(j+\ell)$ in this case, by
$$\frac{ (9 \beta/2) (2p)! (j+\ell)}{
 [(\beta/4)(j+\ell)]^3 }
 = \frac{288 (2p)!}{\beta^2 (j+\ell)^2}.$$
We deduce, 
for $j+\ell \geq (2/\beta)(4p+2)$, 
\begin{align*}\mathbb{E} \left[ \prod_{m=1}^p \left| 1 -  \boldsymbol{\alpha}_{j+\ell-1} \mathrm{e}^{\textnormal{i} \Psi_{j+ \ell - 1} (\theta_m) }\right|^{2s_m} \, \bigg| \mathcal{F}_{j+ \ell-1} \right] 
& \leq 1 + \frac{2}{ \beta(j+\ell)} \left|\sum_{m=1}^{p} s_m \mathrm{e}^{\textnormal{i} \Psi_{j+ \ell - 1} (\theta_m) }\right|^2
 + \frac{K}{(j+\ell)^2}, 
\end{align*}
where $K > 0$ depends only on $\beta$, $p$ and 
the sequence $(s_m)_{1 \leq m \leq p}$. 
The restriction $j+\ell \geq (2/\beta)(4p+2)$ can trivially be removed by suitably adjusting 
the value of $K$. 
We deduce, after adjusting $K$ again, 
\begin{align*}
& \mathbb{E} \left[ \prod_{m=1}^p \left| 1 -  \boldsymbol{\alpha}_{j+\ell-1} \mathrm{e}^{\textnormal{i} \Psi_{j+ \ell - 1} (\theta_m) }\right|^{2s_m} \, \bigg| \mathcal{F}_{j+ \ell-1} \right] 
 \\ & \leq 1 + \frac{2}{ \beta(j+\ell)} \left|\sum_{m=1}^{p} s_m \mathrm{e}^{\textnormal{i} (\Psi_{j} (\theta_m) + (\ell-1) \theta_m) }\right|^2
 +\\
 & \ \ \frac{K}{j+\ell} \left(
 t + \mathbf{1}_{\exists m \in \{1,\dots,p\},
 |\Psi_{j+\ell-1} (\theta_m) - \Psi_j (\theta_m) - (\ell-1)\theta_m| \geq t} 
 \right) +\frac{K}{(j+\ell)^2} 
 \end{align*}
 for any $t > 0$. 
Multiplying this bound by 
the product of $|\Phi^*_{j+\ell-1}(\mathrm{e}^{\textnormal{i} \theta_m})|^{2s_m}$ 
and taking the conditional expectation given
$\mathcal{F}_j$, we deduce 
\begin{align*}
& \mathbb{E} \left[ \prod_{m=1}^p \left|\Phi^*_{j+\ell} (\mathrm{e}^{\textnormal{i} \theta_m})
\right|^{2s_m} \;  \bigg| \mathcal{F}_j \right]
\\ & \leq 
 \left(1 + \frac{2}{ \beta(j+\ell)} \left|\sum_{m=1}^{p} s_m \mathrm{e}^{\textnormal{i} (\Psi_{j} (\theta_m) + (\ell-1) \theta_m) }\right|^2
 +\frac{Kt}{j+\ell} + \frac{K}{(j+\ell)^2}\right) 
 \mathbb{E} \left[ \prod_{m=1}^p \left|\Phi^*_{j+\ell-1} (\mathrm{e}^{\textnormal{i} \theta_m})
\right|^{2s_m} \;  \bigg| \mathcal{F}_j \right]
\\ & + \frac{K}{j+\ell}
\mathbb{E} \left[
 \prod_{m=1}^p \left|\Phi^*_{j+\ell-1} (\mathrm{e}^{\textnormal{i} \theta_m})
\right|^{2s_m}  
\mathbf{1}_{\exists m \in \{1,\dots,p\},
 |\Psi_{j+\ell-1} (\theta_m) - \Psi_j (\theta_m) - (\ell-1)\theta_m| \geq t} \, \bigg| \mathcal{F}_j
\right].
\end{align*}
Taking $t = 1$ and bounding the indicator function by $1$, we get 
in particular, for a possibly different value of $K$ chosen to be larger than $1$, 
\begin{align*}\mathbb{E} \left[ \prod_{m=1}^p \left|\Phi^*_{j+\ell} (\mathrm{e}^{\textnormal{i} \theta_m})
\right|^{2s_m} \;  \bigg| \mathcal{F}_j \right]
& \leq \left(1 + \frac{K}{j+\ell} \right)
\mathbb{E} \left[ \prod_{m=1}^p \left|\Phi^*_{j+\ell-1} (\mathrm{e}^{\textnormal{i} \theta_m})
\right|^{2s_m} \;  \bigg| \mathcal{F}_j \right]
\\ & \leq \left(1 + \frac{1}{j+\ell} \right)^K
\mathbb{E} \left[ \prod_{m=1}^p \left|\Phi^*_{j+\ell-1} (\mathrm{e}^{\textnormal{i} \theta_m})
\right|^{2s_m} \;  \bigg| \mathcal{F}_j \right],
\end{align*}
which by induction on $\ell$, gives 
$$\mathbb{E} \left[ \prod_{m=1}^p \left|\Phi^*_{j+\ell} (\mathrm{e}^{\textnormal{i} \theta_m})
\right|^{2s_m} \;  \bigg| \mathcal{F}_j \right]
\leq \left(\frac{j+\ell+1}{j+1} \right)^K 
\prod_{m=1}^p \left|\Phi^*_{j} (\mathrm{e}^{\textnormal{i} \theta_m})
\right|^{2s_m}
$$
for all $\ell \geq 0$. By the Cauchy-Schwarz inequality and an application of the 
inequality just above to $2s_1, \dots, 2s_p$
instead of $s_1, \dots, s_p$, 
we deduce, for $\ell \geq 1$, 
\begin{align*}
& \mathbb{E} \left[
 \prod_{m=1}^p \left|\Phi^*_{j+\ell-1} (\mathrm{e}^{\textnormal{i} \theta_m})
\right|^{2s_m}  
\mathbf{1}_{\exists m \in \{1,\dots,p\},
 |\Psi_{j+\ell-1} (\theta_m) - \Psi_j (\theta_m) - (\ell-1)\theta_m| \geq t} \, \bigg| \mathcal{F}_j
\right]
\\ & 
\leq \left(\mathbb{E} \left[
 \prod_{m=1}^p \left|\Phi^*_{j+\ell-1} (\mathrm{e}^{\textnormal{i} \theta_m})
\right|^{4s_m}  
 \bigg| \mathcal{F}_j
\right] \right)^{1/2} \times \\
& \ \ \mathbb{P} \left(\exists m \in \{1,\dots,p\},
 |\Psi_{j+\ell-1} (\theta_m) - \Psi_j (\theta_m) - (\ell-1)\theta_m| \geq t \, | \mathcal{F}_j
\right)^{1/2}
\\ & \leq \left(\frac{j+\ell+1}{j+1} \right)^K 
\prod_{m=1}^p \left|\Phi^*_{j} (\mathrm{e}^{\textnormal{i} \theta_m})
\right|^{2s_m}  \sqrt{2 p}  \exp \left( - \frac{t^2 \beta}{16 \log \left(1 + \frac{\beta ( \ell-1)}{1 + \beta j} \right)}\right) \mathbf{1}_{\ell \geq 2}
\\ & \leq \left(\frac{j+\ell+1}{j+1} \right)^K 
\prod_{m=1}^p |\Phi^*_{j} (\mathrm{e}^{\textnormal{i} \theta_m})
|^{2s_m}  \sqrt{2 p}  \exp \left( - \frac{t^2  \beta j}{16 \ell}\right),
\end{align*}
for some $K > 0$ depending only on $\beta$, $p$
and $(s_m)_{1 \leq m \leq p}$, where the second inequality is obtained from a union bound on $m$ and 
Proposition \ref{boundlargefluctuationpsi}, and the last inequality 
is due to the fact that for $\ell \geq 2$, 
$$\log \left(1 + \frac{\beta(\ell-1)}{1 + \beta j} \right) \leq  \log \left(1 + \frac{\beta \ell}{ \beta j} \right) = \log (1 + \ell/j) \leq \ell/j.$$
Now, if we assume $\ell \leq j$, 
$(j+\ell+1)/(j+1)$ is smaller than $2$, and then 
its $K$-th power is bounded by a quantity depending only on $\beta$, $p$
and $(s_m)_{1 \leq m \leq p}$. We then deduce, after adjusting $K$, 
\begin{align*}
& \mathbb{E} \left[ \prod_{m=1}^p \left|\Phi^*_{j+\ell} (\mathrm{e}^{\textnormal{i} \theta_m})
\right|^{2s_m} \;  \bigg| \mathcal{F}_j \right]
\\ & \leq 
 \left(1 + \frac{2}{ \beta(j+\ell)} \left|\sum_{m=1}^{p} s_m \mathrm{e}^{\textnormal{i} (\Psi_{j} (\theta_m) + (\ell-1) \theta_m) }\right|^2
 +\frac{Kt}{j+\ell} + \frac{K}{(j+\ell)^2}\right) 
 \mathbb{E} \left[ \prod_{m=1}^p \left|\Phi^*_{j+\ell-1} (\mathrm{e}^{\textnormal{i} \theta_m})
\right|^{2s_m} \;  \bigg| \mathcal{F}_j \right]
\\ & + \frac{K}{j+\ell}
\prod_{m=1}^p \left|\Phi^*_{j} (\mathrm{e}^{\textnormal{i} \theta_m})
\right|^{2s_m}  \exp \left( - \frac{t^2  \beta j}{16 \ell}\right).
\end{align*}
By induction on $\ell \in \{0,1,\dots, j\}$, we deduce
\begin{align*}
& \mathbb{E} \left[ \prod_{m=1}^p \left|\Phi^*_{j+\ell} (\mathrm{e}^{\textnormal{i} \theta_m})
\right|^{2s_m} \;  \bigg| \mathcal{F}_j \right]
\\ & \leq 
\prod_{r = 1}^{\ell} \left(1 + \frac{2}{ \beta(j+r)} \left|\sum_{m=1}^{p} s_m \mathrm{e}^{\textnormal{i} (\Psi_{j} (\theta_m) + (r-1) \theta_m) }\right|^2
 +\frac{K (t + \mathrm{e}^{-t^2 \beta j/(16 r)})}{j+r} + \frac{K}{(j+r)^2}\right) 
 \prod_{m=1}^p \left|\Phi^*_{j} (\mathrm{e}^{\textnormal{i} \theta_m})
\right|^{2s_m}
\\ & \leq \prod_{m=1}^p \left|\Phi^*_{j} (\mathrm{e}^{\textnormal{i} \theta_m})
\right|^{2s_m} \times \\  
&\ \ \ \exp \left( 
\frac{2}{\beta (j+1)} \sum_{r = 1}^{\ell}
\left|\sum_{m=1}^{p} s_m \mathrm{e}^{\textnormal{i} (\Psi_{j} (\theta_m) + (r-1) \theta_m) }\right|^2 
+ \frac{K \ell (t + \mathrm{e}^{-t^2 \beta j/(16 \ell)})}{j+1}
+ \frac{K \ell}{(j+1)^2}
\right).
\end{align*}
Taking, for $\ell \geq 1$, and then $j \geq 1$, $$t = \sqrt{\frac{ 8 \log ((1+j)/(1+\ell))}{\beta j/\ell}},$$
we get 
$$t + \mathrm{e}^{-t^2 \beta j/(16 \ell)} 
= \sqrt{\frac{ 8 \log ((1+j)/(1+\ell))}{\beta j/\ell}}
+ \mathrm{e}^{- \log ((1+j)/(1+\ell)) / 2},$$
which provides a bound of the form 
\begin{align*}
 &\mathbb{E} \left[ \prod_{m=1}^p \left|\Phi^*_{j+\ell} (\mathrm{e}^{\textnormal{i} \theta_m})
\right|^{2s_m} \;  \bigg| \mathcal{F}_j \right]  \leq \prod_{m=1}^p \left|\Phi^*_{j} (\mathrm{e}^{\textnormal{i} \theta_m})
\right|^{2s_m} \times  \\ &\exp \left( 
\frac{2}{\beta (j+1)} \sum_{r = 1}^{\ell}
\left|\sum_{m=1}^{p} s_m \mathrm{e}^{\textnormal{i} (\Psi_{j} (\theta_m) + (r-1) \theta_m) }\right|^2 
+ \frac{K \ell  (1 + \sqrt{\log((1+j)/(1+\ell))} }{(j+1)\sqrt{(1+j)/(1+\ell)}}
\right),
\end{align*}
the term $K\ell/(j+1)^2$ being absorbed by the previous term after adjusting $K$. 
This bound obviously occurs also for $\ell = 0$. 
If we expand the squared modulus and sum in $r$, we get a double sum, for $1 \leq m_1, m_2 \leq p$, of geometric series. 
Separating the geometric series corresponding to the terms where $m_1 = m_2$ and the series for $m_1 \neq m_2$ gives a bound of the following form:  
\begin{align*}
& \mathbb{E} \left[ \prod_{m=1}^p \left|\Phi^*_{j+\ell} (\mathrm{e}^{\textnormal{i} \theta_m})
\right|^{2s_m} \;  \bigg| \mathcal{F}_j \right]
\\ & \leq \prod_{m=1}^p \left|\Phi^*_{j} (\mathrm{e}^{\textnormal{i} \theta_m})
\right|^{2s_m} \, \exp \left[ 
\frac{2 \ell}{\beta (j+1)}  
 \sum_{m=1}^{p} s_m^2 
+ \frac{K}{j+1} \min \left( \ell, \left(\min_{1 \leq m_1 \neq m_2 \leq p} 
||\theta_{m_1} - \theta_{m_2}|| \right)^{-1} \right)  \dots \right.
\\ & \left.  \dots
+ \frac{K \ell  (1 + \sqrt{\log((1+j)/(1+\ell))} }{(j+1)\sqrt{(1+j)/(1+\ell)}}
\right].
\end{align*}
This completes the proof of the proposition. 
\end{proof}

By iterating Proposition \ref{boundconditionalmoment} we are able to deduce the following.

\begin{prop} \label{boundconditionalmomentsdoubling}
For $p \geq 2$, let $(\theta_m)_{1 \leq m \leq p}\in \mathbb{R}^p$ , $(s_m)_{1 \leq m \leq p}\in \mathbb{R}_+^p$. 
Let 
$$\mu \overset{\textnormal{def}}{=} \min_{1 \leq m_1 \neq m_2 \leq p} 
||\theta_{m_1} - \theta_{m_2}||.$$
Then, for all $j\in \mathbb{N}$ so that $j \geq \mu^{-1}$, 
$$\mathbb{E} \left[ \prod_{m=1}^p \left|\Phi^*_{2j} (\mathrm{e}^{\textnormal{i} \theta_m})\right|^{2s_m} \;  \bigg| \mathcal{F}_j \right]
\leq \prod_{m=1}^p \left|\Phi^*_{j} (\mathrm{e}^{\textnormal{i} \theta_m})
\right|^{2s_m} \, \exp \left(
\frac{2 \log 2}{\beta}  
 \sum_{m=1}^{p} s_m^2 
+ K \left( \frac{\log (1 + j \mu)}{j \mu} \right)^{1/3} \right)
$$
where $K > 0$ depends only on $\beta, p$
and on the sequence $(s_m)_{1 \leq m \leq p}$. 
\end{prop}
\begin{proof}
For $q\in \mathbb{N}$ such that $1 \leq q \leq j$, we consider $j_0, \dots, j_q \in \mathbb{N}$ such that 
$$j = j_0 \leq j_1 \leq \dots \leq j_q = 2j,$$
and 
$$|j_{r + 1} - j_r - j/q| \leq 1$$
for $0 \leq r \leq q-1$. 
For $0 \leq r \leq q-1$, applying Proposition \ref{boundconditionalmoment}
to $j = j_r$, $\ell = j_{r+1} - j_r \leq 2j - j \leq j_r$, and conditioning on $\mathcal{F}_j$ gives 
an inequality of the form
\begin{align*}
& \mathbb{E} \left[ \prod_{m=1}^p \left|\Phi^*_{j_{r+1}} (\mathrm{e}^{\textnormal{i} \theta_m})
\right|^{2s_m} \;  \bigg| \mathcal{F}_j \right]
\\ & \leq \mathbb{E} \left[ \prod_{m=1}^p \left|\Phi^*_{j_r} (\mathrm{e}^{\textnormal{i} \theta_m})
\right|^{2s_m} \, \bigg| \mathcal{F}_j \right] \, \exp \left[ 
\frac{2 (j_{r+1}-j_r)}{\beta j_r}  
 \sum_{m=1}^{p} s_m^2 
+ \frac{K \mu^{-1}}{j} 
+ K q^{-3/2}   \sqrt{\log (1 +q)}    
\right].
\end{align*}
Combining these inequalities for all values of $r$ between $0$ and $q-1$ gives
\begin{align*}
\mathbb{E} \left[ \prod_{m=1}^p \left|\Phi^*_{2j} (\mathrm{e}^{\textnormal{i} \theta_m})
\right|^{2s_m} \;  \bigg| \mathcal{F}_j \right]
&\leq \prod_{m=1}^p \left|\Phi^*_{j} (\mathrm{e}^{\textnormal{i} \theta_m})
\right|^{2s_m} \times \\ 
&\exp \left[ 
\sum_{r = 0}^{q-1}\frac{2 (j_{r+1}-j_r)}{\beta j_r}  
 \sum_{m=1}^{p} s_m^2 
+ \frac{K q }{j \mu} 
+ K    \sqrt{\frac{\log (1 +q)}{q}}    
\right]. 
\end{align*}
We have, for all $x \in [0,1]$, 
$$0 \leq x - \log(1+x) = \frac{x^2}{2} 
- \frac{x^3}{3} + \frac{x^4}{4} - \dots 
\leq \frac{x^2}{2}
$$
and then 
$$0 \leq \frac{j_{r+1}-j_r}{j_r} - \log (j_{r+1}/j_r)
\leq \frac{1}{2} \left(\frac{j_{r+1}- j_r}{j_r} \right)^2 
\leq \frac{1}{2}\left(\frac{j/q + 1}{j} \right)^2
\leq \frac{1}{2}\left(\frac{2j/q}{j} \right)^2
\leq \frac{2}{q^2}.$$
We deduce a bound of the form
\begin{align*}
& \mathbb{E} \left[ \prod_{m=1}^p \left|\Phi^*_{2j} (\mathrm{e}^{\textnormal{i} \theta_m})
\right|^{2s_m} \;  \bigg| \mathcal{F}_j \right] \\ & 
\leq \prod_{m=1}^p \left|\Phi^*_{j} (\mathrm{e}^{\textnormal{i} \theta_m})
\right|^{2s_m} \exp \left[ 
\sum_{r = 0}^{q-1}\frac{2}{\beta} \log (j_{r+1}/j_r)  
 \sum_{m=1}^{p} s_m^2 
+ \frac{K}{q} + \frac{K q }{j \mu} 
+ K    \sqrt{\frac{\log (1 +q)}{q}}    
\right].
\end{align*}
We have $j \mu \geq 1$ by assumption. 
We choose $q = f(j \mu)$ where $f$ is a function 
from $[1, \infty)$ to $\mathbb{Z}$ such that $f(x) = 1$ for 
$x \in [1, \pi]$, $1 \leq f(x) \leq x/\pi$ for $x \geq \pi$ and $f(x)$ is equivalent to $x^{2/3} (\log x)^{1/3}$ when $x \rightarrow \infty$. 
The inequalities satisfied by $f$ guarantee that 
$1 \leq q \leq j$. Indeed, since $\mu$ is always at most $\pi$, we have $1 \leq f(j \mu) \leq j \mu/ \pi \leq j$ if $j \mu \geq \pi$, and $ f(j \mu) = 1 \leq j $ if $j \mu \in [1, \pi]$. 
For $j \mu \geq 2$ sufficiently large, this choice of $q$ gives 
\begin{align*}
    &\frac{1}{q} + \frac{q}{j \mu} 
+  \sqrt{ \frac{\log (1+q)}{q} } 
\\ & \leq \frac{1}{(j \mu)^{2/3} (\log (j \mu))^{1/3}/2} + \frac{ 2 (j \mu)^{2/3} (\log (j \mu))^{1/3} }{j \mu} + \sqrt{ \frac{\log (1+2 (j \mu)^{2/3} (\log (j \mu))^{1/3} )}{(j \mu)^{2/3} (\log (j \mu))^{1/3}/2} },
 \end{align*}
which is dominated by $(j \mu)^{-1/3} (\log (1 + j \mu))^{1/3}$. 
This bound remains true for $j \mu \geq 1$ smaller than a universal constant, since $q = f(j \mu) \in [1, \max(1, j \mu/\pi)]$ takes finitely many values in this case.
This completes the proof of Proposition \ref{boundconditionalmomentsdoubling}, after observing that 
$$\sum_{r = 0}^{q-1} \log (j_{r+1}/j_r) 
= \log (j_q/j_0) = \log 2.$$

\end{proof}

We can now relatively easily deduce the following bound on conditional moments on $\mathcal{F}_j$ for general $\Phi_N^*$ instead of $\Phi_{2j}^*$.

\begin{prop} \label{boundconditionalmomentsfinal}
For $p \geq 2$, let $(\theta_m)_{1 \leq m \leq p}\in \mathbb{R}^p$ 
, $(s_m)_{1 \leq m \leq p}\in \mathbb{R}_+^p$. 
Let 
$$\mu \overset{\textnormal{def}}{=} \min_{1 \leq m_1 \neq m_2 \leq p} 
||\theta_{m_1} - \theta_{m_2}||.$$
Then, for all $N,j\in \mathbb{N}$ so that $N \geq j \geq \mu^{-1}$, 
$$\mathbb{E} \left[ \prod_{m=1}^p \left|\Phi^*_{N} (\mathrm{e}^{\textnormal{i} \theta_m})\right|^{2s_m} \;  \bigg| \mathcal{F}_j \right]
\leq K \left( \frac{N}{j} \right)^{\frac{2}{\beta} \sum_{m=1}^p s_m^2}  \prod_{m=1}^p \left|\Phi^*_{j} (\mathrm{e}^{\textnormal{i} \theta_m})
\right|^{2s_m} 
$$
where $K > 0$ depends only on $\beta, p$
and on the sequence $(s_m)_{1 \leq m \leq p}$. 
\end{prop}
\begin{proof}
If $j \geq \mu^{-1}$, using Proposition \ref{boundconditionalmomentsdoubling}, we deduce, by 
induction on $r \geq 0$,  
$$\mathbb{E} \left[ \prod_{m=1}^p \left|\Phi^*_{2^r j} (\mathrm{e}^{\textnormal{i} \theta_m})\right|^{2s_m} \;  \bigg| \mathcal{F}_j \right]
\leq 
\prod_{m=1}^p \left|\Phi^*_{j} (\mathrm{e}^{\textnormal{i} \theta_m})
\right|^{2s_m} \, \exp \left(
\frac{2 r \log 2}{\beta}  
 \sum_{m=1}^{p} s_m^2 
+ K \sum_{b = 1}^r \left( \frac{\log (1 + 2^b j \mu)}{2^b j \mu} \right)^{1/3} \right). 
$$
Applying Proposition \ref{boundconditionalmoment}
to $2^r j$ instead of $j$ and to $\ell = N - 2^r j$,
and then conditioning with respect to $\mathcal{F}_j$, we deduce, for $2^r j \leq N \leq 2^{r+1} j$, an inequality of the form 
\begin{align*}
\mathbb{E} \left[ \prod_{m=1}^p \left|\Phi^*_{N} (\mathrm{e}^{\textnormal{i} \theta_m})\right|^{2s_m} \;  \bigg| \mathcal{F}_j \right]
&\leq 
  \prod_{m=1}^p \left|\Phi^*_{j} (\mathrm{e}^{\textnormal{i} \theta_m})
\right|^{2s_m} \times \\ &\exp \left(K + 
\frac{2 r \log 2}{\beta}  
 \sum_{m=1}^{p} s_m^2 
+ K \sum_{b = 1}^r \left( \frac{\log (1 + 2^b j \mu)}{2^b j \mu} \right)^{1/3} \right),
\end{align*}
since the exponential factor in Proposition \ref{boundconditionalmoment} is uniformly bounded 
by a quantity depending only on $\beta, p$ and 
$(s_m)_{1 \leq m \leq p}$. 
Now, we are done, since the sum in $b$ is unifomly bounded for $j \mu \geq 1$, and 
$$\exp \left(\frac{2 r \log 2}{\beta}  
 \sum_{m=1}^{p} s_m^2 \right) = (2^{r})^{\frac{2}{\beta} \sum_{m=1}^p s_m^2} 
 \leq (N/j)^{\frac{2}{\beta} \sum_{m=1}^p s_m^2}. $$

\end{proof}

We can now prove the following sharp bound on the joint moments of $(|\Phi^*_{N} (\mathrm{e}^{\textnormal{i} \theta_m})|^{2s_m})_{m=1}^p$ by an inductive argument on the number of factors in the product $\prod_{m=1}^p |\Phi^*_{N} (\mathrm{e}^{\textnormal{i} \theta_m})|^{2s_m}$.

\begin{prop}\label{FinalBoundProp}
Let $(\theta_m)_{1 \leq m \leq p}\in \mathbb{R}^p$ 
, $(s_m)_{1 \leq m \leq p}\in \mathbb{R}_+^p$ . 
Then, for all $N \in \mathbb{N}$, 
$$\mathbb{E} \left[ \prod_{m=1}^p \left|\Phi^*_{N} (\mathrm{e}^{\textnormal{i} \theta_m})\right|^{2s_m} \;  \right]
\leq K N^{\frac{2}{\beta} \sum_{m=1}^p s_m^2}  \prod_{1 \leq m_1 < m_2 \leq p} 
\min(N, |\mathrm{e}^{\textnormal{i} \theta_{m_2}} - \mathrm{e}^{\textnormal{i} \theta_{m_1}} |^{-1})^{\frac{4}{\beta} s_{m_1} s_{m_2}},
$$
where $K > 0$ depends only on $\beta, p$
and on the sequence $(s_m)_{1 \leq m \leq p}$. 
\end{prop}

\begin{proof}
We prove this result by induction on $p \geq 1$. 
The case $p = 1$ is a direct consequence of the fact that 
$$\left|\Phi^*_{N} (\mathrm{e}^{\textnormal{i} \theta_1})\right|^{2s_1}
= \prod_{j=0}^{N-1} \left|1 - \boldsymbol{\alpha}_j\right|^{2 s_1}$$
in distribution, and from the moment estimate
$$\mathbb{E} [|1 - \boldsymbol{\alpha}_j|^{2 s_1}] \leq e^{2s_1^2 /(1 + \beta(j+1))} $$
proven in Proposition 2.5 of \cite{ChhaibiMadauleNajnudel}, which 
implies 
$$\mathbb{E} [|\Phi^*_{N} (\mathrm{e}^{\textnormal{i} \theta_1})|^{2s_1}] \leq \mathrm{e}^{(2s_1^2/\beta) \sum_{j=0}^{N-1} (j+1)^{-1} }
\leq \mathrm{e}^{(2 s_1^2/\beta)(1 + \log N)}.
$$
For $p \geq 2$, let us assume that the proposition holds for $p-1$ points on the unit circle. 
By symmetry, we can assume that 
the closest pair of points among $(\mathrm{e}^{\textnormal{i} \theta_{m}})_{1 \leq m \leq p}$ is given by 
$\theta_{p-1}$ and $\theta_p$. 
Let
$$j = \min(N, \lceil \, ||\theta_{p} - \theta_{p-1} ||^{-1} \, \rceil).$$
If $||\theta_p - \theta_{p+1}|| > 1/N$, we have
$j \geq ||\theta_{p}- \theta_{p-1}||^{-1}$ and then we can apply Proposition \ref{boundconditionalmomentsfinal}, which implies
$$\mathbb{E} \left[ \prod_{m=1}^p \left|\Phi^*_{N} (\mathrm{e}^{\textnormal{i} \theta_m})\right|^{2s_m} \;  \right]
\leq  K \left( \frac{N}{j} \right)^{\frac{2}{\beta} \sum_{m=1}^p s_m^2}  \mathbb{E} \left[ \prod_{m=1}^p \left|\Phi^*_{j} (\mathrm{e}^{\textnormal{i} \theta_m})
\right|^{2s_m} \right].$$
This inequality remains obviously true for $||\theta_p - \theta_{p-1}|| \leq 1/N$, in which case $j =N$. 
We deduce
\begin{align*}
&\mathbb{E} \left[ \prod_{m=1}^p \left|\Phi^*_{N} (\mathrm{e}^{\textnormal{i} \theta_m})\right|^{2s_m} \;  \right]
\leq \\
&K \left( \frac{N}{j} \right)^{\frac{2}{\beta} \sum_{m=1}^p s_m^2}  \mathbb{E} \left[ 
\max\left(\left|\Phi^*_{j} (\mathrm{e}^{\textnormal{i} \theta_{p-1}}) \right|^{2(s_p + s_{p-1})}, \left|\Phi^*_{j} (\mathrm{e}^{\textnormal{i} \theta_{p}}) \right|^{2(s_p + s_{p-1})}\right) \prod_{m=1}^{p-2} \left|\Phi^*_{j} (\mathrm{e}^{\textnormal{i} \theta_m})
\right|^{2s_m} \right].
\end{align*}

Bounding the maximum by a sum and applying the induction hypothesis, we deduce an equality of the form
\begin{align*}& \mathbb{E} \left[ \prod_{m=1}^p \left|\Phi^*_{N} (\mathrm{e}^{\textnormal{i} \theta_m})\right|^{2s_m} \;  \right]
 \leq K \left( \frac{N}{j} \right)^{\frac{2}{\beta} \sum_{m=1}^p s_m^2} 
j^{\frac{2}{\beta} \left( (s_{p-1} + s_p)^2 + \sum_{m=1}^{p-2}  s_m^2 \right)}
 \dots
\\ &
\dots \times
\sum_{r \in \{p-1,p\}}
\prod_{m =1}^{p-2}
\min(j, |\mathrm{e}^{\textnormal{i} \theta_{m}} - \mathrm{e}^{\textnormal{i} \theta_{r}} |^{-1})^{\frac{4}{\beta} s_{m} (s_{p-1} + s_p)}   \prod_{1 \leq m_1 < m_2 \leq p-2} 
\min(j, |\mathrm{e}^{\textnormal{i} \theta_{m_2}} - \mathrm{e}^{\textnormal{i} \theta_{m_1}} |^{-1})^{\frac{4}{\beta} s_{m_1} s_{m_2}}.
\end{align*}
Notice that here, the product in $m$ is equal to $1$ for $p = 2$ and the product in $(m_1, m_2)$ is equal to $1$ for $p \in \{2,3\}$. 
Now, for $1 \leq m \leq p-2$, 
$$|\mathrm{e}^{\textnormal{i} \theta_{m}} - \mathrm{e}^{\textnormal{i} \theta_{p}} | \leq 
|\mathrm{e}^{\textnormal{i} \theta_{m}} - \mathrm{e}^{\textnormal{i} \theta_{p-1}} |
+ |\mathrm{e}^{\textnormal{i} \theta_{p}} - \mathrm{e}^{\textnormal{i} \theta_{p-1}} |
\leq 2 |\mathrm{e}^{\textnormal{i} \theta_{m}} - \mathrm{e}^{\textnormal{i} \theta_{p-1}} |
$$
because $\mathrm{e}^{\textnormal{i} \theta_{p-1}}$
and $\mathrm{e}^{\textnormal{i} \theta_{p}}$ are the closest points among $(\mathrm{e}^{\textnormal{i} \theta_{m}})_{1 \leq m \leq p}$. Similarly, 
$$|\mathrm{e}^{\textnormal{i} \theta_{m}} - \mathrm{e}^{\textnormal{i} \theta_{p-1}} | \leq 2 |\mathrm{e}^{\textnormal{i} \theta_{m}} - \mathrm{e}^{\textnormal{i} \theta_{p}} |.$$
Hence, the ratio between the two terms for $r = p-1$ and $r = p$ is bounded from above and below by a positive quantity depending only on $\beta$ and 
$(s_m)_{1 \leq m \leq p}$. 
In the estimate, we can then replace the sum of the two terms by 
any of their weighted geometric means after adjusting $K$. We deduce 
\begin{align*}& \mathbb{E} \left[ \prod_{m=1}^p \left|\Phi^*_{N} (\mathrm{e}^{\textnormal{i} \theta_m})\right|^{2s_m} \;  \right]
 \leq K \left( \frac{N}{j} \right)^{\frac{2}{\beta} \sum_{m=1}^p s_m^2} 
j^{\frac{2}{\beta} \left( (s_{p-1} + s_p)^2 + \sum_{m=1}^{p-2}  s_m^2 \right)}
 \dots
\\ &
\dots \times
\prod_{r \in \{p-1,p\}}
\prod_{m =1}^{p-2}
\min(j, |\mathrm{e}^{\textnormal{i} \theta_{m}} - \mathrm{e}^{\textnormal{i} \theta_{r}} |^{-1})^{\frac{4}{\beta} s_{m} s_r}   \prod_{1 \leq m_1 < m_2 \leq p-2} 
\min(j, |\mathrm{e}^{\textnormal{i} \theta_{m_2}} - \mathrm{e}^{\textnormal{i} \theta_{m_1}} |^{-1})^{\frac{4}{\beta} s_{m_1} s_{m_2}}
\\ & = K N^{\frac{2}{\beta} \sum_{m=1}^p s_m^2} 
j^{\frac{4}{\beta} s_{p-1} s_p}
\prod_{1 \leq m_1 \leq p-2, m_1 < m_2 \leq p} 
\min(j, |\mathrm{e}^{\textnormal{i} \theta_{m_2}} - \mathrm{e}^{\textnormal{i} \theta_{m_1}} |^{-1})^{\frac{4}{\beta} s_{m_1} s_{m_2}}
\\ & \leq K N^{\frac{2}{\beta} \sum_{m=1}^p s_m^2} 
\min(N,\lceil \, ||\theta_{p} - \theta_{p-1} ||^{-1} \, \rceil)^{\frac{4}{\beta} s_{p-1} s_p}
\prod_{1 \leq m_1 \leq p-2, m_1 < m_2 \leq p} 
\min(N, |\mathrm{e}^{\textnormal{i} \theta_{m_2}} - \mathrm{e}^{\textnormal{i} \theta_{m_1}} |^{-1})^{\frac{4}{\beta} s_{m_1} s_{m_2}}. 
\end{align*}
This proves the desired result, 
since 
$$\lceil \, ||\theta_{p} - \theta_{p-1} ||^{-1} \, \rceil \leq 1 + ||\theta_{p} - \theta_{p-1} ||^{-1} 
\leq 1 + |\mathrm{e}^{\textnormal{i} \theta_{p}}
- \mathrm{e}^{\textnormal{i} \theta_{p-1}}|^{-1} \leq 
3 |\mathrm{e}^{\textnormal{i} \theta_{p}}
- \mathrm{e}^{\textnormal{i} \theta_{p-1}}|^{-1}.
$$
\end{proof}

We can finally prove the main bound from Theorem \ref{MainBound}.

\begin{proof}[Proof of Theorem \ref{MainBound}]
By a change of measure, we can write, for $2\delta\ge \sum_{j=1}^\ell r_j$, so that in particular all exponents are non-negative,
\begin{align*}
\mathbb{E}\left[\left|\mathfrak{q}_N^{\beta,\delta}(\mathrm{e}^{\mathrm{i}x_1/N})\right|^{r_1}\cdots \left|\mathfrak{q}_N^{\beta,\delta}(\mathrm{e}^{\mathrm{i}x_\ell/N})\right|^{r_\ell}\right]=\frac{\mathbb{E}_{\textnormal{C}\beta \textnormal{E}_N}\left[\left|\mathsf{X}_N(1)\right|^{2\delta-\sum_{j=1}^\ell r_j} \left|\mathsf{X}_N(\mathrm{e}^{\mathrm{i}x_1/N})\right|^{r_1}\cdots \left|\mathsf{X}_N(\mathrm{e}^{\mathrm{i}x_\ell/N})\right|^{r_\ell}\right]}{\mathbb{E}_{\textnormal{C}\beta \textnormal{E}_N}\left[\left|\mathsf{X}_N(1)\right|^{2\delta}\right]}.
\end{align*}
Now, recall that the function $\theta \mapsto \mathsf{X}_N(\mathrm{e}^{\mathrm{i}\theta})$ has the same distribution as $\theta \mapsto \Phi^*_{N-1} (\mathrm{e}^{\textnormal{i}\theta })- \mathrm{e}^{\textnormal{i}\theta } \boldsymbol{\eta} \Phi_{N-1} (\mathrm{e}^{\textnormal{i}\theta })$, where $\boldsymbol{\eta}$ is uniformly distributed on $\mathbb{T}$, and also $| \Phi^*_{N-1} (\mathrm{e}^{\textnormal{i}\theta })- \mathrm{e}^{\textnormal{i}\theta } \boldsymbol{\eta} \Phi_{N-1} (\mathrm{e}^{\textnormal{i}\theta })| \le 2 |\Phi^*(\mathrm{e}^{\mathrm{i}\theta})|$.
Then, plugging in the bound from Proposition \ref{FinalBoundProp}, making use of the asymptotics from Lemma \ref{AsymptoticsPartition} for the denominator and using the elementary inequality
\begin{equation*}
N^{-\frac{t_jt_k}{\beta}}\min\left(N,\left|\mathrm{e}^{\mathrm{i}x_j/N}-\mathrm{e}^{\mathrm{i}x_k/N}\right|^{-1}\right)^{\frac{t_jt_k}{\beta}} \le C \frac{1}{1+\left|x_j-x_k\right|^{t_jt_k/\beta}},
\end{equation*}
for some constant $C$ depending only on 
$\beta > 0$, $t_j, t_k \geq 0$, 
we obtain the desired result.
\end{proof}

\bibliographystyle{acm}
\bibliography{References}

\begin{thebibliography}{10}

\bibitem{ArguinBeliusBourgade}
{\sc Arguin, L.-P., Belius, D., and Bourgade, P.}
\newblock Maximum of the characteristic polynomial of random unitary matrices.
\newblock {\em Comm. Math. Phys. 349}, 2 (2017), 703--751.

\bibitem{ABBRS}
{\sc Arguin, L.-P., Belius, D., Bourgade, P., Radziwi\l\l, M., and
  Soundararajan, K.}
\newblock Maximum of the {R}iemann zeta function on a short interval of the
  critical line.
\newblock {\em Comm. Pure Appl. Math. 72}, 3 (2019), 500--535.

\bibitem{ABR1}
{\sc Arguin, L.-P., Bourgade, P., and Radziwi{\l}{\l}, M.}
\newblock The fyodorov-hiary-keating conjecture. i.
\newblock {\em arXiv preprint arXiv:2007.00988\/} (2020).

\bibitem{ABR2}
{\sc Arguin, L.-P., Bourgade, P., and Radziwiłł, M.}
\newblock The fyodorov-hiary-keating conjecture. ii.
\newblock {\em arXiv preprint arXiv:2307.00982\/} (2023).

\bibitem{MoMCbetaE}
{\sc Assiotis, T.}
\newblock On the moments of the partition function of the {$C\beta E$} field.
\newblock {\em J. Stat. Phys. 187}, 2 (2022), Paper No. 14, 22.

\bibitem{RandomAnalytic}
{\sc Assiotis, T.}
\newblock Random entire functions from random polynomials with real zeros.
\newblock {\em Adv. Math. 410\/} (2022), Paper No. 108701, 28.

\bibitem{ABK}
{\sc Assiotis, T., Bailey, E.~C., and Keating, J.~P.}
\newblock On the moments of the moments of the characteristic polynomials of
  {H}aar distributed symplectic and orthogonal matrices.
\newblock {\em Ann. Inst. Henri Poincar\'e{} D 9}, 3 (2022), 567--604.

\bibitem{ABGS}
{\sc Assiotis, T., Bedert, B., Gunes, M.~A., and Soor, A.}
\newblock On a distinguished family of random variables and {P}ainlev\'e{}
  equations.
\newblock {\em Probab. Math. Phys. 2}, 3 (2021), 613--642.

\bibitem{MoMEhrhart}
{\sc Assiotis, T., Eriksson, E., and Ni, W.}
\newblock On the moments of moments of random matrices and {E}hrhart
  polynomials.
\newblock {\em Adv. in Appl. Math. 149\/} (2023), Paper No. 102539, 30.

\bibitem{AGKW}
{\sc Assiotis, T., Gunes, M.~A., Keating, J.~P., and Wei, F.}
\newblock Exchangeable arrays and integrable systems for characteristic
  polynomials of random matrices.
\newblock {\em arXiv preprint arXiv:2407.19233\/} (2024).

\bibitem{AssiotisKeatingMoM}
{\sc Assiotis, T., and Keating, J.~P.}
\newblock Moments of moments of characteristic polynomials of random unitary
  matrices and lattice point counts.
\newblock {\em Random Matrices Theory Appl. 10}, 2 (2021), Paper No. 2150019,
  16.

\bibitem{AKW}
{\sc Assiotis, T., Keating, J.~P., and Warren, J.}
\newblock On the joint moments of the characteristic polynomials of random
  unitary matrices.
\newblock {\em Int. Math. Res. Not. IMRN}, 18 (2022), 14564--14603.

\bibitem{BaileyKeatingMoM}
{\sc Bailey, E.~C., and Keating, J.~P.}
\newblock On the moments of the moments of the characteristic polynomials of
  random unitary matrices.
\newblock {\em Comm. Math. Phys. 371}, 2 (2019), 689--726.

\bibitem{BaileyKeatingZeta}
{\sc Bailey, E.~C., and Keating, J.~P.}
\newblock On the moments of the moments of {$\zeta(1/2+it)$}.
\newblock {\em J. Number Theory 223\/} (2021), 79--100.

\bibitem{BaileyKeatingSurvey}
{\sc Bailey, E.~C., and Keating, J.~P.}
\newblock Maxima of log-correlated fields: some recent developments.
\newblock {\em J. Phys. A 55}, 5 (2022), Paper No. 053001, 76.

\bibitem{BasorChenEhrhardt}
{\sc Basor, E., Chen, Y., and Ehrhardt, T.}
\newblock Painlev\'e{} {V} and time-dependent {J}acobi polynomials.
\newblock {\em J. Phys. A 43}, 1 (2010), 015204, 25.

\bibitem{BasorGeRubinstein}
{\sc Basor, E., Ge, F., and Rubinstein, M.~O.}
\newblock Some multidimensional integrals in number theory and connections with
  the {P}ainlev\'e{} {V} equation.
\newblock {\em J. Math. Phys. 59}, 9 (2018), 091404, 14.

\bibitem{Berestycki}
{\sc Berestycki, N.}
\newblock An elementary approach to {G}aussian multiplicative chaos.
\newblock {\em Electron. Commun. Probab. 22\/} (2017), Paper No. 27, 12.

\bibitem{BorodinDet}
{\sc Borodin, A.}
\newblock Determinantal point processes.
\newblock In {\em The {O}xford handbook of random matrix theory}. Oxford Univ.
  Press, Oxford, 2011, pp.~231--249.

\bibitem{BorodinOlshanski}
{\sc Borodin, A., and Olshanski, G.}
\newblock Infinite random matrices and ergodic measures.
\newblock {\em Comm. Math. Phys. 223}, 1 (2001), 87--123.

\bibitem{BourgadeBulk}
{\sc Bourgade, P., Erd\H~os, L., and Yau, H.-T.}
\newblock Universality of general {$\beta$}-ensembles.
\newblock {\em Duke Math. J. 163}, 6 (2014), 1127--1190.

\bibitem{BourgadeFalconet}
{\sc Bourgade, P., and Falconet, H.}
\newblock Liouville quantum gravity from random matrix dynamics.
\newblock {\em arXiv preprint arXiv:2206.03029\/} (2022).

\bibitem{BympGamburd}
{\sc Bump, D., and Gamburd, A.}
\newblock On the averages of characteristic polynomials from classical groups.
\newblock {\em Comm. Math. Phys. 265}, 1 (2006), 227--274.

\bibitem{ChhaibiMadauleNajnudel}
{\sc Chhaibi, R., Madaule, T., and Najnudel, J.}
\newblock On the maximum of the {${\rm C}\beta {\rm E}$} field.
\newblock {\em Duke Math. J. 167}, 12 (2018), 2243--2345.

\bibitem{ChhaibiNajnudel}
{\sc Chhaibi, R., and Najnudel, J.}
\newblock On the circle, $gmc^\gamma = \varprojlim c\beta e_n$ for $\gamma =
  \sqrt{\frac{2}{\beta}}, $ $( \gamma \leq 1 )$.
\newblock {\em arxiv:1904.00578\/} (2019).

\bibitem{CNN}
{\sc Chhaibi, R., Najnudel, J., and Nikeghbali, A.}
\newblock The circular unitary ensemble and the {R}iemann zeta function: the
  microscopic landscape and a new approach to ratios.
\newblock {\em Invent. Math. 207}, 1 (2017), 23--113.

\bibitem{ClaeysFahs}
{\sc Claeys, T., and Fahs, B.}
\newblock Random matrices with merging singularities and the {P}ainlev\'e{} {V}
  equation.
\newblock {\em SIGMA Symmetry Integrability Geom. Methods Appl. 12\/} (2016),
  Paper No. 031, 44.

\bibitem{ClaeysForkelKeating}
{\sc Claeys, T., Forkel, J., and Keating, J.~P.}
\newblock Moments of moments of the characteristic polynomials of random
  orthogonal and symplectic matrices.
\newblock {\em Proc. A. 479}, 2270 (2023), Paper No. 20220652, 22.

\bibitem{ClaeysGlesnerMinakovYang}
{\sc Claeys, T., Glesner, G., Minakov, A., and Yang, M.}
\newblock Asymptotics for averages over classical orthogonal ensembles.
\newblock {\em Int. Math. Res. Not. IMRN}, 10 (2022), 7922--7966.

\bibitem{ClaeysKrasovsky}
{\sc Claeys, T., and Krasovsky, I.}
\newblock Toeplitz determinants with merging singularities.
\newblock {\em Duke Math. J. 164}, 15 (2015), 2897--2987.

\bibitem{CFKRS}
{\sc Conrey, J.~B., Farmer, D.~W., Keating, J.~P., Rubinstein, M.~O., and
  Snaith, N.~C.}
\newblock Integral moments of {$L$}-functions.
\newblock {\em Proc. London Math. Soc. (3) 91}, 1 (2005), 33--104.

\bibitem{RandomDets}
{\sc Dal~Borgo, M., Hovhannisyan, E., and Rouault, A.}
\newblock Mod-{G}aussian convergence for random determinants.
\newblock {\em Ann. Henri Poincar\'e 20}, 1 (2019), 259--298.

\bibitem{Fahs}
{\sc Fahs, B.}
\newblock Uniform asymptotics of {T}oeplitz determinants with
  {F}isher-{H}artwig singularities.
\newblock {\em Comm. Math. Phys. 383}, 2 (2021), 685--730.

\bibitem{ForresterNuclearPhys}
{\sc Forrester, P.~J.}
\newblock Selberg correlation integrals and the {$1/r^2$} quantum many-body
  system.
\newblock {\em Nuclear Phys. B 388}, 3 (1992), 671--699.

\bibitem{ForresterNuclearPhysAdd}
{\sc Forrester, P.~J.}
\newblock Addendum to: ``{S}elberg correlation integrals and the {$1/r^2$}
  quantum many-body system'' [{N}uclear {P}hys.\ {B} {\bf 388} (1992), no.\ 3,
  671--699; {MR}1201273 (94e:33030)].
\newblock {\em Nuclear Phys. B 416}, 1 (1994), 377--385.

\bibitem{ForresterConj}
{\sc Forrester, P.~J.}
\newblock Integration formulas and exact calculations in the
  {C}alogero-{S}utherland model.
\newblock {\em Modern Phys. Lett. B 9}, 6 (1995), 359--371.

\bibitem{ForresterBook}
{\sc Forrester, P.~J.}
\newblock {\em Log-gases and random matrices}, vol.~34 of {\em London
  Mathematical Society Monographs Series}.
\newblock Princeton University Press, Princeton, NJ, 2010.

\bibitem{ForresterDifferential}
{\sc Forrester, P.~J.}
\newblock Differential identities for the structure function of some random
  matrix ensembles.
\newblock {\em J. Stat. Phys. 183}, 2 (2021), Paper No. 33, 28.

\bibitem{forrester2025dualities}
{\sc Forrester, P.~J.}
\newblock Dualities in random matrix theory.
\newblock {\em arXiv preprint arXiv:2501.07144\/} (2025).

\bibitem{ForresterRahman}
{\sc Forrester, P.~J., and Rahman, A.~A.}
\newblock Relations between moments for the {J}acobi and {C}auchy random matrix
  ensembles.
\newblock {\em J. Math. Phys. 62}, 7 (2021), Paper No. 073302, 22.

\bibitem{ForresterRains}
{\sc Forrester, P.~J., and Rains, E.~M.}
\newblock A {F}uchsian matrix differential equation for {S}elberg correlation
  integrals.
\newblock {\em Comm. Math. Phys. 309}, 3 (2012), 771--792.

\bibitem{FyodorovBouchaud}
{\sc Fyodorov, Y.~V., and Bouchaud, J.-P.}
\newblock Freezing and extreme-value statistics in a random energy model with
  logarithmically correlated potential.
\newblock {\em J. Phys. A 41}, 37 (2008), 372001, 12.

\bibitem{FyodorovGnutzmannKeating}
{\sc Fyodorov, Y.~V., Gnutzmann, S., and Keating, J.~P.}
\newblock Extreme values of {CUE} characteristic polynomials: a numerical
  study.
\newblock {\em J. Phys. A 51}, 46 (2018), 464001, 22.

\bibitem{FyodorovHiaryKeating}
{\sc Fyodorov, Y.~V., Hiary, G.~A., and Keating, J.~P.}
\newblock Freezing transition, characteristic polynomials of random matrices,
  and the riemann zeta function.
\newblock {\em Phys. Rev. Lett. 108\/} (Apr 2012), 170601.

\bibitem{FyodorovKeating}
{\sc Fyodorov, Y.~V., and Keating, J.~P.}
\newblock Freezing transitions and extreme values: random matrix theory, and
  disordered landscapes.
\newblock {\em Philos. Trans. R. Soc. Lond. Ser. A Math. Phys. Eng. Sci. 372},
  2007 (2014), 20120503, 32.

\bibitem{FyodorovSimm}
{\sc Fyodorov, Y.~V., and Simm, N.~J.}
\newblock On the distribution of the maximum value of the characteristic
  polynomial of {GUE} random matrices.
\newblock {\em Nonlinearity 29}, 9 (2016), 2837--2855.

\bibitem{Harper}
{\sc Harper, A.~J.}
\newblock On the partition function of the riemann zeta function, and the
  fyodorov--hiary--keating conjecture.
\newblock {\em arXiv preprint arXiv:1906.05783\/} (2019).

\bibitem{HughesKeatingOConnell}
{\sc Hughes, C.~P., Keating, J.~P., and O'Connell, N.}
\newblock On the characteristic polynomial of a random unitary matrix.
\newblock {\em Comm. Math. Phys. 220}, 2 (2001), 429--451.

\bibitem{JohanssonDet}
{\sc Johansson, K.}
\newblock Random matrices and determinantal processes.
\newblock In {\em Mathematical statistical physics}. Elsevier B. V., Amsterdam,
  2006, pp.~1--55.

\bibitem{KRRR}
{\sc Keating, J.~P., Rodgers, B., Roditty-Gershon, E., and Rudnick, Z.}
\newblock Sums of divisor functions in {$\Bbb{F}_q[t]$} and matrix integrals.
\newblock {\em Math. Z. 288}, 1-2 (2018), 167--198.

\bibitem{KeatingSnaithSympOrth}
{\sc Keating, J.~P., and Snaith, N.~C.}
\newblock Random matrix theory and {$L$}-functions at {$s=1/2$}.
\newblock {\em Comm. Math. Phys. 214}, 1 (2000), 91--110.

\bibitem{KeatingSnaith}
{\sc Keating, J.~P., and Snaith, N.~C.}
\newblock Random matrix theory and {$\zeta(1/2+it)$}.
\newblock {\em Comm. Math. Phys. 214}, 1 (2000), 57--89.

\bibitem{KeatingWong}
{\sc Keating, J.~P., and Wong, M.~D.}
\newblock On the critical-subcritical moments of moments of random
  characteristic polynomials: a {GMC} perspective.
\newblock {\em Comm. Math. Phys. 394}, 3 (2022), 1247--1301.

\bibitem{KillipNenciu}
{\sc Killip, R., and Nenciu, I.}
\newblock Matrix models for circular ensembles.
\newblock {\em International Mathematics Research Notices 2004}, 50 (01 2004),
  2665--2701.

\bibitem{KillipStoiciu}
{\sc Killip, R., and Stoiciu, M.}
\newblock Eigenvalue statistics for {CMV} matrices: from {P}oisson to clock via
  random matrix ensembles.
\newblock {\em Duke Math. J. 146}, 3 (2009), 361--399.

\bibitem{LambertNajnudel}
{\sc Lambert, G., and Najnudel, J.}
\newblock Subcritical multiplicative chaos and the characteristic polynomial of
  the {C}$\beta ${E}.
\newblock {\em arXiv preprint arXiv:2407.19817\/} (2024).

\bibitem{LambertPaquetteAiry}
{\sc Lambert, G., and Paquette, E.}
\newblock Strong approximation of gaussian $\beta$-ensemble characteristic
  polynomials: the edge regime and the stochastic airy function.
\newblock {\em arXiv preprint arXiv:2009.05003\/} (2020).

\bibitem{LambertPaquetteBulk}
{\sc Lambert, G., and Paquette, E.}
\newblock Bulk asymptotics of the gaussian $\beta$-ensemble characteristic
  polynomial.
\newblock {\em arXiv preprint arXiv:2508.01458\/} (2025).

\bibitem{Levin}
{\sc Levin, B.~J.}
\newblock {\em Distribution of zeros of entire functions}, revised~ed., vol.~5
  of {\em Translations of Mathematical Monographs}.
\newblock American Mathematical Society, Providence, RI, 1980.
\newblock Translated from the Russian by R. P. Boas, J. M. Danskin, F. M.
  Goodspeed, J. Korevaar, A. L. Shields and H. P. Thielman.

\bibitem{LiValko}
{\sc Li, Y., and Valk\'o, B.}
\newblock Operator level limit of the circular {J}acobi {$\beta$}-ensemble.
\newblock {\em Random Matrices Theory Appl. 11}, 4 (2022), Paper No. 2250043,
  41.

\bibitem{NajnudelZeta}
{\sc Najnudel, J.}
\newblock On the extreme values of the {R}iemann zeta function on random
  intervals of the critical line.
\newblock {\em Probab. Theory Related Fields 172}, 1-2 (2018), 387--452.

\bibitem{NikulaSaksmanWebb}
{\sc Nikula, M., Saksman, E., and Webb, C.}
\newblock Multiplicative chaos and the characteristic polynomial of the {CUE}:
  the {$L^1$}-phase.
\newblock {\em Trans. Amer. Math. Soc. 373}, 6 (2020), 3905--3965.

\bibitem{Okounkov}
{\sc Okounkov, A.}
\newblock On n-point correlations in the log-gas at rational temperature.
\newblock {\em arXiv preprint hep-th/9702001\/} (1997).

\bibitem{Osada1}
{\sc Osada, H.}
\newblock Dirichlet form approach to infinite-dimensional {W}iener processes
  with singular interactions.
\newblock {\em Comm. Math. Phys. 176}, 1 (1996), 117--131.

\bibitem{Osada2}
{\sc Osada, H.}
\newblock Non-collision and collision properties of {D}yson's model in infinite
  dimension and other stochastic dynamics whose equilibrium states are
  determinantal random point fields.
\newblock In {\em Stochastic analysis on large scale interacting systems},
  vol.~39 of {\em Adv. Stud. Pure Math.} Math. Soc. Japan, Tokyo, 2004,
  pp.~325--343.

\bibitem{PaquetteZeitouni1}
{\sc Paquette, E., and Zeitouni, O.}
\newblock The maximum of the {CUE} field.
\newblock {\em Int. Math. Res. Not. IMRN}, 16 (2018), 5028--5119.

\bibitem{PaquetteZeitouni2}
{\sc {Paquette}, E., and {Zeitouni}, O.}
\newblock {The extremal landscape for the C$\beta$E ensemble}.
\newblock {\em arXiv e-prints\/} (Sept. 2022), 118pp.

\bibitem{qu2025pair}
{\sc Qu, Y., and Valk{\'o}, B.}
\newblock On the pair correlation function of the $sine_\beta$ process.
\newblock {\em arXiv preprint arXiv:2509.15446\/} (2025).

\bibitem{Remy}
{\sc Remy, G.}
\newblock The {F}yodorov-{B}ouchaud formula and {L}iouville conformal field
  theory.
\newblock {\em Duke Math. J. 169}, 1 (2020), 177--211.

\bibitem{RhodesVargas}
{\sc Rhodes, R., and Vargas, V.}
\newblock Gaussian multiplicative chaos and applications: a review.
\newblock {\em Probab. Surv. 11\/} (2014), 315--392.

\bibitem{Suzuki1}
{\sc Suzuki, K.}
\newblock On the ergodicity of interacting particle systems under number
  rigidity.
\newblock {\em Probab. Theory Related Fields 188}, 1-2 (2024), 583--623.

\bibitem{Suzuki2}
{\sc Suzuki, K.}
\newblock Curvature bound of {D}yson {B}rownian motion.
\newblock {\em Comm. Math. Phys. 406}, 7 (2025), Paper No. 154, 56.

\bibitem{ValkoViragCarousel}
{\sc Valk\'o, B., and Vir\'ag, B.}
\newblock Continuum limits of random matrices and the {B}rownian carousel.
\newblock {\em Invent. Math. 177}, 3 (2009), 463--508.

\bibitem{ValkoViragOperators}
{\sc Valk\'o, B., and Vir\'ag, B.}
\newblock The {$\rm Sine_\beta$} operator.
\newblock {\em Invent. Math. 209}, 1 (2017), 275--327.

\bibitem{ValkoVirag}
{\sc Valk\'o, B., and Vir\'ag, B.}
\newblock The many faces of the stochastic zeta function.
\newblock {\em Geom. Funct. Anal. 32}, 5 (2022), 1160--1231.

\bibitem{Webb}
{\sc Webb, C.}
\newblock The characteristic polynomial of a random unitary matrix and
  {G}aussian multiplicative chaos---the {$L^2$}-phase.
\newblock {\em Electron. J. Probab. 20\/} (2015), no. 104, 21.

\bibitem{Widom}
{\sc Widom, H.}
\newblock Toeplitz determinants with singular generating functions.
\newblock {\em Amer. J. Math. 95\/} (1973), 333--383.

\bibitem{ForresterWitte}
{\sc Witte, N.~S., and Forrester, P.~J.}
\newblock Gap probabilities in the finite and scaled {C}auchy random matrix
  ensembles.
\newblock {\em Nonlinearity 13}, 6 (2000), 1965--1986.

\end{thebibliography}

\bigskip 

\noindent{\sc School of Mathematics, University of Edinburgh, James Clerk Maxwell Building, Peter Guthrie Tait Rd, Edinburgh EH9 3FD, U.K.}\newline
\href{mailto:theo.assiotis@ed.ac.uk}{\small theo.assiotis@ed.ac.uk}

\bigskip

\noindent{\sc School of Mathematics, University of Bristol, Fry Building, Woodland Road, Bristol, BS8 1UG, U.K.}\newline
\href{mailto:joseph.najnudel@bristol.ac.uk}{\small joseph.najnudel@bristol.ac.uk}
\end{document}